\documentclass[12pt,draftcls,onecolumn]{IEEEtran}
\usepackage{cite}
\usepackage{url}
\usepackage{amsmath,amssymb,amsfonts}
\usepackage{multirow}
\usepackage{amsthm}
\usepackage{graphicx}
\usepackage{textcomp}
\usepackage{epsfig}
\usepackage{algorithm}
\usepackage{algorithmic}
\usepackage{epstopdf}
\usepackage{enumerate}
\usepackage{float}
\usepackage{graphics,graphicx}
\usepackage{subfigure}
\usepackage{array}
\usepackage{mathtools}
\usepackage{color}
\usepackage{mathtools}
\DeclareGraphicsExtensions{.png}
\graphicspath{{./Pics/}}
\usepackage{xspace}
\usepackage{tcolorbox}
\usepackage{dashbox}
\usepackage{bm} 
\usepackage{caption}
\usepackage{cite}
\usepackage{amsmath,amssymb,amsfonts}
\usepackage{algorithmic}
\usepackage{graphicx}
\usepackage{textcomp}

\usepackage{cite}
\usepackage{url}
\usepackage{amsmath,amssymb,amsfonts}
\usepackage{multirow}
\usepackage{algorithmic}
\usepackage{textcomp}
\usepackage{epsfig}
\usepackage{algorithmic}
\usepackage{epstopdf}
\usepackage{float}
\usepackage{caption}  
\usepackage{booktabs} 
\usepackage{pifont}
\usepackage{array}
\usepackage{mathtools}
\usepackage[justification=centering]{caption}
\usepackage{hyperref}
\usepackage{cleveref}

\newcommand{\mb}{\mathbf}

\renewcommand{\Re}{{\mathbb{R}}}
\newcommand{\bx}{{\mathbf{x}}}
\newcommand{\bB}{{\mathbf{B}}}
\newcommand{\by}{{\mathbf{y}}}
\newcommand{\bd}{{\mathbf{d}}}
\newcommand{\bG}{{\mathbf{G}}}
\newcommand{\bL}{{\mathbf{L}}}
\newcommand{\bv}{{\mathbf{v}}}
\newcommand{\be}{{\mathbf{e}}}
\newcommand{\bz}{{\mathbf{z}}}
\newcommand{\bw}{{\mathbf{w}}}
\newcommand{\bq}{{\mathbf{q}}}
\newcommand{\bI}{{\mathbf{I}}}
\newcommand{\bK}{{\mathbf{K}}}

\newcommand{\bR}{{\mathbf{R}}}
\newcommand{\bJ}{{\mathbf{J}}}
\newcommand{\bQ}{{\mathbf{Q}}}
\newcommand{\bg}{{\mathbf{g}}}
\newcommand{\bs}{{\mathbf{s}}}

\newtheorem{lemma}{Lemma}
\newtheorem{theorem}{Theorem}
\newtheorem{definition}{Definition}
\newtheorem{remark}{Remark}
\newtheorem{assumption}{Assumption}

\newtheorem{corollary}{Corollary}

\usepackage[justification=centering]{caption}
\allowdisplaybreaks[4]
\begin{document}
	
	\title{\LARGE \bf Distributed Nonconvex Optimization with Double Privacy Protection and Exact Convergence}
\author{Zichong Ou, Dandan Wang, Zixuan Liu, and Jie Lu
\thanks{Zichong Ou, Dandan Wang and Jie Lu are with the School of Information Science and Technology, Shanghaitech University, 201210 Shanghai, China. J. Lu is also with the Shanghai Engineering Research Center of Energy Efficient and Custon AI IC, 201210 Shanghai, China. Email: {\tt\small ouzch, wangdd2, lujie@shanghaitech.edu.cn}.}
\thanks{Zixuan Liu is with the Engineering and Technology Institute (ENTEG), University of Groningen, 9747 AG Groningen, the Netherlands. Email: {\tt\small zixuan.liu@rug.nl}.}
}
	\maketitle
	
	\begin{abstract}
Motivated by the pervasive lack of privacy protection in existing distributed nonconvex optimization methods, this paper proposes a decentralized proximal primal-dual algorithm enabling double protection of privacy ($\text{DPP}^2$) for minimizing nonconvex sum-utility functions over multi-agent networks, which ensures zero leakage of critical local information during inter-agent communications. We develop a two-tier privacy protection mechanism that first merges the primal and dual variables by means of a variable transformation, followed by embedding an additional random perturbation to further obfuscate the transmitted information. We theoretically establish that $\text{DPP}^2$ ensures differential privacy for local objectives while achieving exact convergence under nonconvex settings. Specifically, $\text{DPP}^2$ converges sublinearly to a stationary point and attains a linear convergence rate under the additional Polyak-{\L}ojasiewicz (P-{\L}) condition. Finally, a numerical example demonstrates the superiority of $\text{DPP}^2$ over a number of state-of-the-art algorithms, showcasing the faster, exact convergence achieved by $\text{DPP}^2$ under the same level of differential privacy.
\end{abstract}



\section{Introduction}
Decentralized optimization has garnered considerable attention recently. In real-world scenarios, a vast majority of optimization problems exhibit \textit{nonconvex} characteristics. These problems include but are not limited to distributed reinforcement learning \cite{molzahn2017survey}, dictionary learning \cite{Wai2015ACD} and wireless resource management \cite{lee2019deep}. Moreover, with the increasing scale of such problems, distributed nonconvex optimization techniques are becoming progressively urgent to develop, which employ a multi-agent network to enable cooperative optimization, and only allow interactions among neighboring agents. This paper studies the distributed optimization problem
\begin{equation}
   \min _{x \in \mathbb{R}^{d}} f(x)= \sum_{i=1}^{N} f_{i}(x) \label{p1}
\end{equation}
over an $N$-node multi-agent network, where the global objective function $f(x)$ is the sum of the nonconvex and smooth local objectives $f_1,\dots,f_N$, each associated with a node.

To address this problem, a collection of distributed nonconvex optimization algorithms have emerged, including primal gradient-based methods \cite{zeng_nonconvex_2018,di2016next,sun2016sonata} and primal-dual methods \cite{hong2016convergence,mancino2023decentralized,yi2022sublinear,sun2018distributed,sun2019distributed,hong2017prox,alghunaim2022unified,yi2021linear}. Specifically, \cite{zeng_nonconvex_2018} shows that the well-known Decentralized Gradient Descent (DGD) and Proximal DGD (Prox-DGD) \cite{nedic2009distributed} asymptotically converge to the set of stationary solutions for nonconvex objectives, and \cite{di2016next,sun2016sonata,hong2016convergence,mancino2023decentralized,yi2022sublinear,sun2018distributed,sun2019distributed,hong2017prox,alghunaim2022unified,yi2021linear} improve the convergence rate to a sublinear rate of $\mathcal{O}(1/K)$ (where $K$ denotes the number of iterations). Moreover, under the additional Polyak-{\L}ojasiewicz (P-{\L}) condition, \cite{yi2022sublinear,yi2021linear} are shown to converge to the global optimum at a linear rate of $\mathcal{O}(\theta^K)$ (where $\theta\in(0,1)$). 

Despite their satisfactory convergence performance, the aforementioned algorithms heavily rely on the communication of local information to achieve consensus, which inadvertently lead to privacy leakage of sensitive data (including local decisions, local objective functions and their gradients). Existing approaches \cite{zeng_nonconvex_2018,di2016next,sun2016sonata,hong2016convergence,mancino2023decentralized,yi2022sublinear,sun2018distributed,sun2019distributed,hong2017prox,alghunaim2022unified,yi2021linear} typically require nodes to share their local decisions with neighboring agents, potentially exposing private information. Furthermore, gradient-tracking-based methods \cite{di2016next,sun2016sonata} inherently expose gradient information over iterations, creating additional vulnerabilities. Of particular concern is that local decisions often contain highly sensitive data, such as personal medical records \cite{KIM201495} and precise locations of sensor nodes in surveillance networks \cite{Zhang2019TIFS}. Moreover, in multi-robot coordination systems \cite{Zhang2017ADMMBP}, even gradient information can expose movement directions and operational patterns, posing significant security risks. In addition, the frequent exchanges of model parameters (i.e., decision variables) may lead to the disclosure of the raw dataset \cite{Carlini2018TheSS}. 

\begin{table*}[!t]
    \renewcommand{\arraystretch}{1.3} 
    \caption{Comparison to state-of-the-art algorithms with differential privacy. Here, for $x_a,x_b\in\mathbb{R}^d$ and $i_0=1,\dots,N$, we define the differentiable functions $f_{i_0}^{(h)}:\mathbb{R}^d\rightarrow \mathbb{R},h=1,2$ with gradients $\nabla f_{i_0}^{(h)}$ and $\Delta g_{i_0}^{(h)}=\nabla f_{i_0}^{(h)}(x_a)-\nabla f_{i_0}^{(h)}(x_b)$. We denote $K$ as the number of iterations, and $\theta\in(0,1)$.}
    \centering
    \scriptsize
    \begin{tabular}{|c|c|c|c|c|c|c|}
        \hline
        \multicolumn{1}{|c|}{\multirow{2}{*}{Algorithm}}& Problem & DP & Extra    & Diminishing  & Exact & Convergence \\
        &type&guarantee&conditions&stepsize/noise& convergence & rate\\
        \hline
        PrivSGP-VR \cite{zhu2024privsgp}& nonconvex & $(\epsilon,\delta)$-DP & {bounded $\|\nabla f_i-\nabla f\|$}  & stepsize & \ding{55} & {$\mathcal{O}(1/\sqrt{K})$} \\
        \hline
        DIFF2 \cite{DIFF2} & nonconvex & $(\epsilon,\delta)$-DP & bounded $\|\nabla f_i\|$ & stepsize & \ding{55} & $\mathcal{O}(1/\sqrt{K})$ \\
        \hline
        \cite{wang2023decentralized}& nonconvex & ($\epsilon,\delta$)-DP  &  bounded $\|\nabla f_i\|$& stepsize &\ding{51} & asymptotic \\
        \hline
        \cite{wang2024robust}& convex & $\epsilon$-DP & bounded $\|\nabla f_i\|$ & stepsize &\ding{51} & asymptotic \\
        \hline
        DMSP\cite{Huang2015Diff}& strongly convex & $\epsilon$-DP  & bounded $\|\nabla f_i\|$ & stepsize,noise &\ding{55} &asymptotic\\
        \hline
        \multirow{2}{*}{DiaDSP\cite{Ding2022TAC}}& \multirow{2}{*}{strongly convex} & \multirow{2}{*}{$\epsilon$-DP}& bounded $\|\nabla f^{(1)}_{i_0}-\nabla f^{(2)}_{i_0}\|$ & \multirow{2}{*}{noise} & \multirow{2}{*}{\ding{55}} & \multirow{2}{*}{$\mathcal{O}(\theta^K)$}  \\
        & &&$\Delta g_{i_0}^{(1)}=\Delta g_{i_0}^{(2)}$&&&\\
        \hline
        \multirow{2}{*}{eDP-TN \cite{Yuan2024DistributedNG}} & \multirow{2}{*}{strongly convex} & \multirow{2}{*}{$\epsilon$-DP} & bounded $\|\nabla f^{(1)}_{i_0}-\nabla f^{(2)}_{i_0}\|$ & \multirow{2}{*}{noise} & \multirow{2}{*}{\ding{51}} & \multirow{2}{*}{$\mathcal{O}(\theta^K)$}\\
        & &&$\Delta g_{i_0}^{(1)}=\Delta g_{i_0}^{(2)}$&&&\\
        \hline
        \multirow{2}{*}{PPDC\cite{XIE2025112338}}& nonconvex & \multirow{2}{*}{$\epsilon$-DP}& \multirow{2}{*}{bounded $\|\nabla f_i\|$} & \multirow{2}{*}{noise} & \ding{55} & $\mathcal{O}(1/K)$ \\
        &P-{\L} condition & & & & \ding{55} &$\mathcal{O}(\theta^K)$\\
        \hline
        \multirow{2}{*}{\textbf{This paper}}& nonconvex & \multirow{2}{*}{$\epsilon$-DP} & \multirow{2}{*}{bounded $\|\nabla f^{(1)}_{i_0}-\nabla f^{(2)}_{i_0}\|$}  &  \multirow{2}{*}{noise}& \ding{51}& $\mathcal{O}(1/K)$ \\
        & P-{\L} condition& & & & \ding{51} & $\mathcal{O}(\theta^K)$\\
        \hline
    \end{tabular}
    \label{tab:my_label}
\end{table*}

To preserve local information, differential privacy (DP) has received significant attention in recent works. The core mechanism of DP involves injecting carefully designed noise into transmitted information, thereby preventing eavesdroppers from inferring private data based on their observations. In decentralized learning \cite{zhu2024privsgp,DIFF2}, ($\epsilon,\delta$)-DP is commonly adopted, where $\epsilon$ quantifies the privacy guarantee against distinguishing outputs from adjacent datasets, allowing a $\delta$ probability of failure. However, due to the accumulation of noise and the use of stochastic gradients, these methods can only guarantee sublinear convergence of $\mathcal{O}(1/\sqrt{K})$ to a neighborhood of the optimal solutions. While \cite{wang2023decentralized} achieves exact convergence via vanishing stepsizes, it sacrifices convergence rate, only ensuring asymptotic convergence. 

For stricter privacy requirements (such as protecting sensitive medical or financial data \cite{wang2022decentralized}), $\epsilon$-DP ($\delta=0$) is particularly suitable. Yet, static noises under $\epsilon$-DP lead to a accumulative explosion of parameters \cite{mironov2017renyi}, prompting existing $\epsilon$-DP methods \cite{Huang2015Diff,Ding2022TAC,xie_compressed_2023,XIE2025112338,huang_differential_2024,Yuan2024DistributedNG} to employ decaying noises for convergence. Under strong convexity, the methods in \cite{Ding2022TAC,xie_compressed_2023,huang_differential_2024} achieve linear convergence to a neighborhood of the optimum. Meanwhile, \cite{XIE2025112338} achieves sublinear convergence to a neighborhood of stationarity for nonconvex objectives and linear convergence to the global optimum under the additional P-{\L} condition. Notably, as is proven in \cite{huang_differential_2024}, gradient-tracking algorithms cannot achieve $\epsilon$-DP and exact convergence simultaneously, which limits the works in \cite{Huang2015Diff,Ding2022TAC,xie_compressed_2023,XIE2025112338,huang_differential_2024} to suboptimal convergence only. The recent studies \cite{wang2024robust,Yuan2024DistributedNG} achieve exact convergence with $\epsilon$-DP. However, \cite{wang2024robust} relies on convexity and \cite{Yuan2024DistributedNG} requires more stringent strong convexity as well as additional assumptions, as is stated in Table~\ref{tab:my_label}. 

In this paper, we design a \emph{\underline{D}ecentralized \underline{P}roximal \underline{P}rimal-dual algorithm enabling \underline{D}ouble \underline{P}rivacy \underline{P}rotection} ($\text{DPP}^2$) for addressing a class of nonconvex optimization problems over multi-agent networks. In $\text{DPP}^2$, each node minimizes an augmented-Lagrangian-like function comprising a linearized objective function and a proximal term, which is followed by a dual ascent step. 
We then introduce an encryption strategy, called \textit{double privacy protection}, which effectively protects local private information from being eavesdropped by adversaries during local communications.
The main contributions of this paper are highlighted as follows:
\begin{itemize}
    \item[1)] \textbf{A novel privacy protection strategy:} We propose a novel two-tier privacy protection strategy for our proposed algorithm, referred to as double privacy protection. The first-tier privacy protection integrates dual variables into transmissions of both \textit{local decisions} and \textit{gradients}, ensuring the security of them during local exchanges. The second-tier privacy protection incorporates decaying Laplace noises into transmission for preserving \textit{local objectives}. The two tiers of protection complement each other, leading to strong privacy and convergence guarantees as is stated in Table~\ref{tab:my_label}.
    
    \item[2)] \textbf{Differential privacy guarantee:} We prove that the proposed double privacy protection strategy achieves $\epsilon$-DP for protecting local objectives from being eavesdropped by adversaries. This is more stringent than $(\epsilon,\delta)$-DP achieved by \cite{zhu2024privsgp,DIFF2,wang2023decentralized}.
    
    \item[3)] \textbf{Exact convergence:} In addition to the $\epsilon$-DP guarantee, $\text{DPP}^2$ also ensures exact convergence for nonconvex problems. This improves the suboptimality results in \cite{Huang2015Diff,Ding2022TAC,xie_compressed_2023,XIE2025112338,huang_differential_2024} (which also employ decaying Laplace noises) and extends the implementation of \cite{wang2024robust} and \cite{Yuan2024DistributedNG} to nonconvex problems.

    \item[4)] \textbf{Fast convergence under mild conditions:} $\text{DPP}^2$ attains a $\mathcal{O}(1/K)$ sublinear rate of convergence to a stationary point for the nonconvex problem, outperforming the existing algorithms with privacy protection that only guarantee asymptotic convergence \cite{wang2023decentralized,Huang2015Diff,gade2018privacy,lou2017privacy,wang2022decentralized}. Moreover, a linear convergence rate is achieved to reach the global optimum under the P-{\L} condition, which is a relaxation of strong convexity assumed in \cite{Ding2022TAC,DingTie2018CDC,Huang2015Diff,Yuan2024DistributedNG}. We also weaken the assumption of bounded gradients in \cite{wang2023decentralized,Huang2015Diff,wang2024robust,DIFF2} to $\delta$-adjacency\footnote{The parameter $\delta$ in $\delta$-adjacency is a distinct concept from the $\delta$ in the ``classic'' $(\epsilon,\delta)$-DP, and there is no relation between the two $\delta$ symbols.} (stated in Definition~\ref{def Adjacency}) and require milder assumptions than the methods in \cite{Ding2022TAC,Yuan2024DistributedNG}.
\end{itemize}

The rest of paper is organized as follows: Section~\ref{section: problem formulation} formulates the distributed optimization problem. Section~\ref{section algorithm design} introduces the development of $\text{DPP}^2$. Section~\ref{section:convergence analysis} provides its convergence results and Section~\ref{section:privacy} analyzes its differential privacy guarantee. Moreover, Section~\ref{section:simulation} compares $\text{DPP}^2$ with related works via a numerical example. Finally, Section~\ref{section conclusion} concludes the paper.

\textbf{Notation:} Given any differentiable function $f$, $\nabla f$ denotes the gradient of $f$. Let $\operatorname{Null}(\cdot)$ represent the null space of a given matrix argument; additionally, we define $\mathbf{1}_n$ ($\mathbf{0}_n$) and $\mathbf{I}_n$ ($\mathbf{O}_n$) as the column one (zero) vector and identity matrix (zero matrix) of dimension $n$, respectively. We use $\langle \cdot,\cdot \rangle$ to denote the Euclidean inner product, $\otimes$ for the Kronecker product, and $\|\cdot\|$ for the $\ell_2$ norm. For any two matrices $\mb{A},\mb{B}\in \mathbb{R}^{d\times d}$, $\mb{A}\succ \mb{B}$ means $\mb{A}-\mb{B}$ is positive definite, and $\mb{A} \succeq \mb{B}$ means $\mb{A}-\mb{B}$ is positive semi-definite. Let $\lambda_i^\mb{A}$ denote the $i$-th largest eigenvalue of $\mb{A}$, and $\mb{A}^{\dagger}$ the Moore-Penrose inverse of $\mb{A}$. If $\mb{A}$ is symmetric and $\mb{A} \succeq \mathbf{O}_d$, for $\mb{x}\in \mathbb{R}^d$, $\|\mb{x}\|^2_\mb{A}:=\mb{x}^{\mathsf{T}}\mb{A}\mb{x}$. For a probability space $\Omega$ and a random variable $\xi\in \Omega$, denote $\mathbb{P}(\xi|\Omega)$ as the probability $\xi$ on $\Omega$ and $\mathbb{E}(\xi)$ as the expectation of $\xi$. For a given parameter $\theta$, $\operatorname{Lap}(\theta)$ denotes the Laplace distribution with probability density function 
$f_L(x,\theta)=\frac{1}{2\theta}e^{-\frac{|x|}{\theta}}$.
	
	\section{Problem Formulation}\label{section: problem formulation}
This section formulates the distributed optimization problem and presents the definitions pertinent to differential privacy.

\subsection{Distributed Optimization Problem}
Consider a network of $N$ nodes, which is modeled as a connected, undirected graph $\mathcal{G} = (\mathcal{V},\mathcal{E})$, where the vertex set $\mathcal{V} = \{1, \dots, N\}$ is the set of $N$ nodes and the edge set $\mathcal{E} \subseteq \{\{i,j\}|i,j \in \mathcal{V},i \neq j\}$ describes the underlying interactions among the nodes. Through the network, each node $i \in \mathcal{V}$ only communicates with its neighboring nodes in $\mathcal{N}_i=\{j\in\mathcal{V}:\{i,j\}\in\mathcal{E}\}$. All the nodes collaboratively solve problem~\eqref{p1},
where the local objective $f_i : \mathbb{R}^d \rightarrow \mathbb{R}$ is differentiable and privately owned by node $i$. Next, We impose the following assumptions on problem~\eqref{p1}:
\begin{assumption}\label{assumption smooth}
The local objective function $f_i:\Re^{d}\rightarrow\Re$ is $M_i$-smooth for some $M_i\geq 0$, i.e.,
\begin{equation*}
\|\nabla f_i(x) - \nabla f_i(y)\|\leq M_i\|x-y\|,\quad\forall x,y\in \Re^{d}.
\end{equation*}
\end{assumption}
\begin{assumption}\label{assumption optimum}
The function $f(x)$ is lower bounded by $f^*:=\inf_{x} f(x)$ over $x\in\mathbb{R}^d$, i.e., $
	f(x) \geq f^* > -\infty.$
\end{assumption}

Assumptions~\ref{assumption smooth} and~\ref{assumption optimum} are commonly adopted in existing works on distributed nonconvex optimization \cite{alghunaim2022unified, sun2018distributed, sun2019distributed,mancino2023decentralized, yi2021linear, hong2016convergence,yi2022sublinear,zhu2024privsgp,DIFF2,wang2023decentralized,XIE2025112338}. 

To solve problem~\eqref{p1} over the graph $\mathcal{G}$, we let each node $i \in \mathcal{V}$ maintain a local estimate $x_i \in \mathbb{R}^d$ of the global decision $x \in \mathbb{R}^d$ in problem~\eqref{p1}, and define 
\begin{equation*}
\tilde{f}(\mathbf{x}) := \sum _{i \in \mathcal{V}} f_i(x_i),\quad \mb{x} = (x_1^{\mathsf{T}}, \dots, x_N^{\mathsf{T}})^{\mathsf{T}} \in \Re ^{Nd}.
\end{equation*}
It has been shown in \cite{mokhtari2016decentralized} that problem~\eqref{p1} can be equivalently transformed  into 
\begin{equation}
\begin{aligned}
\underset{\mathbf{x} \in \mathbb{R}^{N d}}{\operatorname{minimize}}\,\,  \tilde{f}(\mb{x}) \quad
\operatorname{subject\;to}\,\,  \mb{L}^{\frac{1}{2}} \mb{x}=0,
\end{aligned} \label{p2}
\end{equation}
where $\mb{L}\in \mathbb{S}^{Nd}$ satisfies the following assumption.
\begin{assumption}\label{assumption GL}
	The symmetric matrix $\mb{L}\in \mathbb{S}^{Nd}$ is positive semidefinite (i.e., $\mathbf{L}\succeq \mathbf{O}_{Nd}$) and has null space $\operatorname{Null}(\mb{L}) = \mathcal{S} := \{\mb{x} \in \Re ^{Nd}|x_1 = \cdots = x_N\}$.
\end{assumption}

Assumption~\ref{assumption GL} aligns with the consensus constraint in \eqref{p2} and is prevalent in the literature, e.g., \cite{zeng_nonconvex_2018,di2016next,alghunaim2022unified,mancino2023decentralized,yi2022sublinear,yi2021linear,wang2023decentralized,wang2024robust,Huang2015Diff,Ding2022TAC,XIE2025112338,scaman2017optimal}.

Note that problem~\eqref{p1} and \eqref{p2} share the same optimal value. Clearly, under Assumption~\ref{assumption smooth}, $\tilde{f}$ is $\bar{M}-$smooth, i.e.,
\begin{equation}\label{smooth tilde f}
    \|\nabla \tilde{f}(\bx)-\tilde{f}(\by)\|\leq \bar{M}\|\bx-\by\|, \quad \forall \bx,\by\in\mathbb{R}^{Nd},
\end{equation}
where $\bar{M}=\max\{M_1,M_2,\dots,M_N\}$.

\subsection{Differential Privacy}
In the communication network, each node transmits local information to its neighbors, which may suffer from information leakage. As the potential attacker has the access to all communication channels, all the information available to the attacker is collected in the observation ${\mathcal{O}}$. To measure the privacy level, we introduce the following definitions on differential privacy.

\begin{definition}\label{def Adjacency}
	($\delta$-\textbf{Adjacency} \cite{DingTie2018CDC,Huang2015Diff}): Given $\delta>0$, two function sets $F^{(1)}=\{f_i^{(1)}\}_{i=1}^N$ and $F^{(2)}=\{f_i^{(2)}\}_{i=1}^N$ are said to be $\delta$-adjacent if there exists $i_{0}$ such that $f_{i}^{(1)}=f_{i}^{(2)}$ for $ i\neq i_0$ and
	\begin{equation}\label{ineq: def delta adjacency}
        \operatorname{Dis}(f_{i_0}^{(1)},f_{i_0}^{(2)})\overset{\triangle}{=}\sup_{x\in\mathbb{R}^d}\|\nabla f_{i_0}^{(1)}(x)-\nabla f_{i_0}^{(2)}(x)\|\leq \delta.
    \end{equation}
\end{definition}

Building on the concept of ``classic'' adjacency on datasets (e.g. \cite{DIFF2,zhu2024privsgp,wang2023decentralized}), we additionally stipulate that the difference between two datasets, measured by a certain metric, should not exceed $\delta$ under a certain metric.
This definition is commonly adopted in the field of distributed optimization \cite{Ding2022TAC,DingTie2018CDC,wang2024robust,huang_differential_2024,Yuan2024DistributedNG,XIE2025112338}. It relaxes the standard assumption of bounded gradients, i.e., $\|\nabla f_i(x_i)\|\leq C,\forall i\in\mathcal{V}$ (e.g., \cite{wang2023decentralized,Huang2015Diff,DIFF2,XIE2025112338}).
To see the relationship, when we consider that $\|\nabla f_i(x_i^k)\|\leq C, \forall k=1,\dots, K$ with $C=\frac{\delta}{2}$, and then we derive $\|\nabla f^{(1)}_{i_0}(x)-\nabla f^{(2)}_{i_0}(x)\|\leq \|\nabla f^{(1)}_{i_0}(x)\|+\|\nabla f^{(2)}_{i_0}(x)\|\leq 2C=\delta$. 
Thus, with the bounded-gradient condition above, $\delta$-adjacency reduces to the ``classic'' notion of adjacency.

\begin{definition}\label{def DP}
	(\textbf{$\epsilon$-Differential Privacy} \cite{DingTie2018CDC,Huang2015Diff}): Given $\delta,\epsilon>0$, for any $\delta$-adjacent function sets $F^{(1)}$ and $F^{(2)}$ and any observation $\mathcal{O}$, a distributed algorithm is said to be $\epsilon$-differentially private if
	$$
	\mathbb{P}(F^{(1)}|\mathcal{O})\leq e^{\epsilon}\mathbb{P}(F^{(2)}|\mathcal{O}),
	$$
	where $\mathbb{P}(F^{(h)}|\mathcal{O}),h=1,2$ is the conditional probability which denotes the probability of inferring $F^{(h)}$ from observation $\mathcal{O}$.
\end{definition}

Intuitively, differential privacy measures how difficult it is for an adversary to distinguish between two adjacent function sets merely by an observation and smaller privacy budget $\epsilon$ means that the two function sets are more indistinguishable based on the observation $\mathcal{O}$. 

Note that the $\epsilon$-Differential Privacy is a more strict than ($\epsilon,\delta$)-Differential Privacy (($\epsilon,\delta$)-DP), as is adopted in \cite{zhu2024privsgp,DIFF2,wang2023decentralized}, which allows for a negligible probability $\delta$ of failure. In this paper, we specifically consider the definition of $\epsilon$-DP as it is particularly well-aligned with scenarios demanding both exact convergence and rigorous privacy guarantees.

\section{Algorithm Development}\label{section algorithm design}

In this section, we develop a distributed algorithm for solving the nonconvex optimization problem~\eqref{p2} (and equivalently, problem~\eqref{p1}), which intends to protect the information privacy of each node while maintaining exact convergence.

To deal with the nonconvex objective function, we first consider the Augmented Lagrangian (AL) function $\operatorname{AL}(\mb{x},\mb{v})=\tilde{f}(\mb{x})+(\mb{v})^{\mathsf{T}} \mb{L}^{\frac{1}{2}}\mb{x}+\frac{\rho}{2}\|\mb{x}\|^2_\mb{L}$,
where $\mb{v}=(v_1^{\mathsf{T}}, \dots, v_N^{\mathsf{T}})^{\mathsf{T}}\in \Re ^{Nd}$ denotes the Lagrangian multiplier and $\rho>0$ is the penalty parameter. We then present the following primal-dual paradigm: Starting from any $\mb{x}^0,\mb{v}^0\in \mathbb{R}^{Nd}$, for each $k \geq 0$,
\begin{align}
	\label{argmin approximation primal}  \mb{x}^{k+1}=&\underset{\mathbf{x} \in \mathbb{R}^{N d}}{\arg \min}\;\tilde{f}(\mb{x}^k) + \langle \nabla \tilde{f}(\mb{x}^k), \mb{x}-\mb{x}^k \rangle +\langle\mb{v}^{k}, \mb{L}^{\frac{1}{2}} \mb{x}\rangle+\frac{\rho}{2}\|\mb{x}\|_{\mb{L}}^{2}+ \frac{1}{2} \|\mb{x}-\mb{x}^k\|^2_{\mb{B}}, \\ 
	\label{argmin approximation dual} {\mb{v}}^{k+1}=&{\mb{v}}^{k}+\rho \mb{L}^{\frac{1}{2}} \mb{x}^{k},
\end{align}
where $\bx^k$ and ${\mb{v}}^k$ are the primal and dual variables at iteration $k$. In \eqref{argmin approximation primal}, we linearize $\tilde{f}(\mb{x})$ at $\mb{x}^k$ as $\tilde{f}(\mb{x}^k)+\langle\nabla \tilde{f}(\mb{x}^k),\mb{x}-\mb{x}^k\rangle$ and embed a proximal term $\frac{1}{2}\|\mb{x}-\mb{x}^{k}\|^2_{\mb{B}}$ with $\mb{B}\in \mathbb{S}^{Nd}$ into the AL function. Moreover, \eqref{argmin approximation dual} emulates a dual ascent step, and the corresponding estimate ``dual gradient" is obtained by evaluating the constraint residual at $\mb{x}^{k}$. Here, we impose a condition on $\bB$ to satisfy $\mb{B}+\rho \mb{L}\succ \mb{O}_{Nd}$, which ensures the well-definedness and unique existence of $\mb{x}^{k+1}$ in \eqref{argmin approximation primal}. Then, the first-order optimality condition of \eqref{argmin approximation primal} gives
\begin{equation}\label{first-order optimality condition}
    \nabla \tilde{f}(\mb{x}^k) + \mb{L}^{\frac{1}{2}} \mb{v}^k + \rho \mb{L} \mb{x}^{k+1}+ \mb{B}(\mb{x}^{k+1} - \mb{x}^k)=0.
\end{equation}
By letting	
\begin{equation}\label{G=B+rhoL-1}
    \mb{G}:= (\mb{B} \!+\! \rho \mb{L})^{-1},
\end{equation}
we rewrite \eqref{argmin approximation primal} as
\begin{equation}
	\mb{x}^{k+1} = \mb{x}^k - \bG(\nabla \tilde{f}(\mb{x}^k) + \mb{L}^{\frac{1}{2}} \mb{v}^k + \rho \mb{L} \mb{x}^k). \label{x^k+1 vk}
\end{equation}

Note that due to the weight matrices $\mathbf{L}^{\frac{1}{2}}$ and $\mathbf{L}$ in \eqref{x^k+1 vk} and \eqref{argmin approximation dual}, our method in its current form cannot be executed in the distributed way. Moreover, to compute $\bL\bx^k,\bL^{\frac{1}{2}}\bv^k$ and $\bL^{\frac{1}{2}}\bx^k$ in \eqref{x^k+1 vk} and \eqref{argmin approximation dual} over $\bG$, the nodes have to share their local portions in $\bx^k$ and $\bv^k$, which risks information leakage. Below we address these issues by first introducing our two-tier privacy protection strategy.

\subsection{First-tier Privacy Protection}\label{section:first-tier}

To formalize the first-tier privacy protection, we apply the following variable transformations
\begin{equation}\label{variable transformation}
    \bq^k= \bL^{\frac{1}{2}}\bv^k, \quad \mb{d}^k = \frac{1}{\rho}(\mb{L}^{\frac{1}{2}})^{\dagger} \mb{v}^k,
\end{equation}
which allows us to substitute $\mb{L}^{\frac{1}{2}} \mb{v}^k$ in \eqref{x^k+1 vk} with $ \eta^k\bq^k+\rho \mb{L}(1-\eta^k)\mb{d}^k$ for some $\eta^k\in(0,1)$. 
This substitution necessitates that $\bq^k,\mathbf{d}^k\in \mathcal{S}^{\perp}\, \forall k \geq 0$, where $\mathcal{S}^{\perp} := \{\mb{x}\in \Re ^{Nd}| \, x_1 + \cdots + x_N = \mb{0}\}$ is the orthogonal complement of $\mathcal{S}$, and can be trivially satisfied by initializing $\bq^0,\mb{d}^0 \in \mathcal{S}^{\perp}$, or simply $\bq^0=\bd^0=0$. Then, starting from arbitrary $\mb{x}^0\in \mathbb{R}^{Nd}$, for any $k \geq 0$, we rewrite \eqref{argmin approximation primal}--\eqref{argmin approximation dual} as
\begin{align} 
    \label{tildeyk} &\by^k = \bx^k+(1-\eta^k)\bd^k,\\
    \label{zk original}&\bz^k = \nabla \tilde{f}(\mb{x}^k)+\eta^k\bq^k+\rho \mb{L} \by^k,\\
    \label{xk+1 original} &\mb{x}^{k+1} = \mb{x}^k - \mb{G} \bz^k,\\ 
    \label{qk+1 original} &{\mb{d}}^{k+1}=\eta^k{\mb{d}}^{k}+{\mb{y}}^{k},\quad
     \bq^{k+1}=\eta^k\bq^k+\rho \mb{L}\by^k.
\end{align}

Notably, the sequence $\{\eta^k\}$ should be predetermined as an input to the algorithm. Each element of this sequence can be randomly generated within the interval $(0,1)$, or alternatively, one may simply set $\eta^k=\eta$ where $\eta\in(0,1)$. 

We also note from \eqref{G=B+rhoL-1} that the condition $\bB+\rho\bL\succ\mathbf{O}_{Nd}$ is equivalent to $\bG\succ\mathbf{O}_{Nd}$. In our implementation, we leverage this by directly constructing a positive definite matrix $\bG$ in the update \eqref{xk+1 original}, thereby avoiding the explicit construction of $\bB$ and the expensive computation of $(\bB+\rho \bL)^{-1}$ required in \eqref{x^k+1 vk}. In addition, the weight matrix $\bL$ and $\bG$ serve the purpose of information propagation in \eqref{tildeyk}--\eqref{qk+1 original}. Moreover, we let the matrix $\bG$ as follows:
\begin{equation}
	\mb{G}=\alpha \mb{I}_{Nd}-\beta \mb{L}, \label{def G}
\end{equation}
with $\alpha>0$ and $0<\beta<\alpha/\lambda_1^{\mb{L}}$, so that $\mb{G}\succ \mb{O}_{Nd}$, and $\mb{G}$ and $\mb{L}$ are commutative in matrix multiplication, i.e., $\mb{G}\mb{L}=\mb{L}\mb{G}$.

To implement the proposed algorithm in a distributed manner, we divide $\bx^k,\bd^k,\by^k,\bz^k$ as $\bx^k=((x_1^k)^{\mathsf{T}},\dots,(x_N^k)^{\mathsf{T}})^{\mathsf{T}}$, $\bd^k=((d_1^k)^{\mathsf{T}},\dots,(d_N^k)^{\mathsf{T}})^{\mathsf{T}}$, $\by^k=((y_1^k)^{\mathsf{T}},\dots,(y_N^k)^{\mathsf{T}})^{\mathsf{T}}$ and $\bz^k=((z_1^k)^{\mathsf{T}},\dots,(z_N^k)^{\mathsf{T}})^{\mathsf{T}}$, and let each node $i$ maintain $x_i^k,d_i^k,y_i^k$ and $z_i^k$.
To meet Assumption~\ref{assumption GL}, we choose 
\begin{equation*}
    \mb{L}= \mb{P} \otimes \mb{I}_d,
\end{equation*}
where $\mb{P} \in \mathbb{S}^{N}$ satisfies $\mb{P}\succeq\mb{O}_N$ with a neighbor-sparse structure, i.e., the off-diagonal entry $p_{ij}$ is zero if nodes $i$ and $j$ are disconnected (i.e., ${i,j} \notin \mathcal{E}$). As shown in \cite{Wu2020AUA}, such a matrix $\mb{P}$ can be determined in a fully decentralized manner by the nodes without central coordination. We can determine $\mb{P}$ as a graph Laplacian matrix and it can be executed in a communication step through the network (detailed in Section~\ref{section:distributed implementation}).

With the above settings, each node $i$ does not directly transmit $x_i^k$ or $\nabla f_i(x_i^k)$ but merges $(1-\eta^k)d_i^k$ and $\eta^k q_i^k$, respectively, during the communication procedure, thereby \textit{preventing eavesdropping on local decisions and gradients}.

Note that the randomness of $\eta^k$ has no impact on the update of $\bx^{k+1}$, as is analyzed in Section~\ref{section:convergence analysis}. Therefore, one cannot observe the same $\mathcal{O}$ (in Definition~\ref{def DP}) generated by $\text{DPP}^2$ with different sequences of $\eta^k$, and thus the first-tier privacy protection lies beyond the reach of the standard differential privacy (DP) analysis and cannot by itself ensure the confidentiality of \textit{local objectives} (or dataset). To tackle this issue, we develop the second-tier privacy protection scheme.



\subsection{Second-tier Privacy Protection} \label{subsection second-tier}
To further enhance data privacy and enable differential privacy for the local data, we incorporate the perturbation variables $\mb{e}^{k}=((e_1^k)^{\mathsf{T}},\dots,(e_N^k)^{\mathsf{T}})^{\mathsf{T}},\bw^k=((w_1^k)^{\mathsf{T}},\dots,(w_N^k)^{\mathsf{T}})^{\mathsf{T}}\in \mathbb{R}^{Nd}$ into the transmission of $\nabla \tilde{f}(\bx^k)$ and $\bx^k$, respectively. Using \eqref{def G}, we rewrite \eqref{tildeyk}--\eqref{qk+1 original} as
\begin{align}
	\label{yk new}&\by^{k} = \mb{x}^k+(1-\eta^k)\mb{d}^k+\bw^k, \\
    \label{zk new}&\bz^k  = \nabla \tilde{f}(\mathbf{x}^k)+\eta^k \mb{q}^k+\rho\mb{L}\by^{k}+\be^k,\\
	\bx^{k+1}=&\bx^k+\bw^k-\alpha\underbrace{(\nabla\tilde{f}(\bx^k)+\eta^k\bq^k+\rho \bL \by^k)}_{\bz^k-\be^k}+\beta \bL\underbrace{(\nabla\tilde{f}(\bx^k)+\eta^k\bq^k+\rho \bL \by^k+\be^k)}_{\bz^k},\label{xk+1 new}\\
	\mathbf{d}^{k+1}  =& \eta^k\mb{d}^k+\mb{y}^k,\quad
	\bq^{k+1} =\eta^k \bq^{k} + \rho \bL\by^k,\label{qk+1 new}
\end{align}
where, in \eqref{yk new} and \eqref{zk new}, the perturbations $\bw^k$ and $\be^k$ are embedded into the transmission of $\by^k$ and $\bz^k$, respectively. Notably, $\be^k$ is only merged into the transmission of $\bz^k$, resulting in the update of \eqref{xk+1 new}.

\textbf{Generation of perturbations}: To guarantee convergence, the perturbations need to vanish along iterations, and thus can be generated by incorporating Laplace noise, i.e., $e_i^k\sim \operatorname{Lap}(\theta_{e,i}^k)$ and $w_i^k\sim \operatorname{Lap}(\theta_{w,i}^k)$. Here, $\theta_{e,i}^k=r_i^k u_{e,i}$ and $\theta_{w,i}^k=r_i^k u_{w,i}$ with $u_{e,i},u_{w,i}>0$ and the noise decay rate $r_i\in(0,1)$. Under such setting, our proposed algorithm is able to \textit{guarantee differential privacy for local objectives}, as will be shown in Section~\ref{section:privacy}.

The above diminishing noise is also adopted in prior works \cite{Huang2015Diff,Ding2022TAC,xie_compressed_2023,XIE2025112338,huang_differential_2024,Yuan2024DistributedNG}. In addition to the protection for local objectives, it also safeguard local decisions and gradients. However, this privacy protection gradually weakens as the noise variance diminishes over iterations. This drawback can be addressed by our proposed first-tier privacy protection mechanism, which integrates dual variables into the transmission process of local decisions and gradients.

\subsection{Interplay Between Two Tiers of Protection}
The above two tiers of privacy protection complement each other in the following way.
Note that the first-tier privacy protection fundamentally differs from most existing privacy strategies that adopt stochastic noises (e.g., \cite{zhu2024privsgp,DIFF2,wang2023decentralized,wang2022decentralized,mironov2017renyi,Huang2015Diff,Ding2022TAC,xie_compressed_2023,XIE2025112338,huang_differential_2024,Yuan2024DistributedNG
}). It integrates the decisions/gradients with the dual variables, preventing the eavesdroppers from inferring local private information like $x_i^k$ and $\nabla f_i(x_i^k)$. Due to its design essence, the first-tier protection does not change the value of the primal variables, so that it cannot be evaluated by the standard DP.

On the other hand, the second-tier protection employs Laplace noises in transmission, so that it guarantees $\epsilon$-DP for local objectives (detailed in Section~\ref{section:privacy}). However, it suffers from gradually losing privacy protection of local decisions and gradients with vanishing noises that are often imposed to guarantee $\epsilon$-DP (e.g., \cite{Huang2015Diff,Ding2022TAC,xie_compressed_2023,XIE2025112338,huang_differential_2024,Yuan2024DistributedNG}). This issue can indeed be overcome by our first-tier protection, as it successfully obfuscates the eavesdroppers' observations. 

To summarize, both tiers of privacy protection are necessary. As will be shown shortly, they \textit{jointly guarantee both DP and exact convergence, maintaining protection even when noise variance approaches zero}. This advantage cannot be achieved by the state-of-the-art distributed nonconvex optimization methods.



\subsection{Distributed Implementation}\label{section:distributed implementation}
    We now illustrate the distributed implementation of the updates \eqref{yk new}--\eqref{qk+1 new} over graph $\mathcal{G}$. We consider the same choices of $\bL$ and $\bG$ as in Section~\ref{section:first-tier}. During each iteration $k$, every node $i$ exchanges encrypted local data $x_i^k+(1-\eta^k)d_i^k+w_i^k$ and $\nabla f_i(x_i^k)+\eta^k q_i^k+e_i^k+\rho\sum_{j \in \mathcal{N}_i \cup \{i\}} p_{ij} {y}_j^k$ with its neighbors. The implementation of \eqref{yk new}--\eqref{qk+1 new} can be distributed to each node $i$ as is shown in Algorithm~\ref{algrithm DPP2}.
\floatname{algorithm}{Algorithm}
\begin{algorithm}[htbp]
        \renewcommand{\thealgorithm}{1}
	\caption{$\text{DPP}^2$}    
	\label{algrithm DPP2}               
	\begin{algorithmic}[1]
	  \STATE \textbf{Input:} $\rho,\alpha,\beta,K>0$, $\{\eta^k\}_{k=0,\dots,K}\in(0,1)$, $\mb{P}\succeq\mathbf{O}_N$.
		\STATE \textbf{Initialization:} Each node $i \in \mathcal{V}$ sets $d_i^0=q_i^0=0$ and arbitrary $x_i^0\in\mathbb{R}^d$.
		\FOR{$k=0,\dots,K$} 
			\FOR{each node $i \in \mathcal{V}$} 
                    \STATE Generate $w_i^k\sim \operatorname{Lap}(\theta_{w,i}^k)$, $e_i^k\sim \operatorname{Lap}(\theta_{e,i}^k)$.
                    \STATE Compute ${y}_i^k =x_i^k+(1-\eta^k)d_i^k+w_i^k$ and send it to every neighbor $j \in \mathcal{N}_i$.
                    \STATE Compute ${z}_i^k = \nabla f_i(x_i^k)+\eta^k q_i^k+\rho \sum_{j \in \mathcal{N}_i \cup \{i\}} p_{ij} {y}_j^k+e_i^k$ and send it to every neighbor $j \in \mathcal{N}_i$.
                    \STATE Update the primal variable $x_i^{k+1} = x_i^k+w_i^k \!-\! \alpha ({z}_i^k-e_i^k)+\beta \sum_{j \in \mathcal{N}_i \cup \{i\}} p_{ij} {z}_j^k$.
                    \STATE Update the dual variable $d_i^{k+1}=\eta^k d_i^k+ y_i^k$.
                    \STATE Update the dual variable $q_i^{k+1}=\eta^k q_i^k+\rho \sum_{j \in \mathcal{N}_i \cup \{i\}} p_{ij} {y}_j^k.$
			\ENDFOR
		\ENDFOR
	\end{algorithmic}
\end{algorithm}

Note that both the updates of $z_i^k$ (in Line 7 of Algorithm. \ref{algrithm DPP2}) and $q_i^{k+1}$ (in Line 10 of Algorithm. \ref{algrithm DPP2}) contain the aggregation term $\sum_{j \in \mathcal{N}_i \cup \{i\}} p_{ij} {y}_j^k$. Therefore, the update of $q_i^{k+1}$ does not entail any extra communication expenses. Due to the information merging and variable perturbations, each node is able to preserve the privacy of its local objectives, decisions as well as their gradients during the communication phase. 
Accordingly, we refer to Algorithm~\ref{algrithm DPP2} as \emph{\underline{D}istributed \underline{P}roximal \underline{P}rimal-dual algorithm with \underline{D}ouble \underline{P}rotection of \underline{P}rivacy}, referred to as $\text{DPP}^2$.

	\section{Convergence Analysis}\label{section:convergence analysis}
This section provides the convergence analysis of $\text{DPP}^2$ under various nonconvex settings.

We first construct an equivalent form of \eqref{yk new}--\eqref{qk+1 new}. Let $\bq^0=\bd^0=\mathbf{0}$ so that $\bq^0=\rho\bL\bd^0$. Since the variable changes of $\bq^k= \bL^{\frac{1}{2}}\bv^k$ and $\mb{d}^k = \frac{1}{\rho}(\mb{L}^{\frac{1}{2}})^{\dagger} \mb{v}^k$ in \eqref{variable transformation} imply $\bq^k=\rho\bL\bd^k$, together with \eqref{qk+1 new}, we have $\bq^{k+1}=\rho\bL(\bx^k+\bd^k+\bw^k)=\rho\bL\bd^{k+1}$. Then, due to \eqref{def G}, we conclude by induction that \eqref{yk new}--\eqref{qk+1 new} are equivalent to 
\begin{align}
	\mathbf{x}^{k+1} =& \bx^k\!+\!\bw^k\!-\!\bG(\nabla\tilde{f}(\bx^k)+\bq^k+\rho\bL(\bx^k+\bw^k))\!+\beta\bL \be^k,\label{xk+1 final}\\
	\bq^{k+1}=&\bq^k+\rho\bL(\bx^k+\bw^k).\label{qk+1 final}
\end{align}
As a result, it can be seen that the parameter $\eta^k$ does not impact the convergence of $\text{DPP}^2$. With the equivalent form \eqref{xk+1 final}--\eqref{qk+1 final}, we establish the convergence results of $\text{DPP}^2$ under a variety of conditions. To this end, we introduce the following notations:
\begin{align}\label{def variable}
    &\mb{K} = (\mb{I}_N - \frac{1}{N} \mb{1}_N \mb{1}_N ^{\mathsf{T}}) \otimes \mb{I}_d,\quad \mb{J} = \frac{1}{N}\mb{1}_N \mb{1}_N ^{\mathsf{T}} \otimes \mb{I}_d,\notag \\
    & \bar{x}^k = \frac{1}{N}(\mb{1}_N^{\mathsf{T}} \otimes \mb{I}_d) \mb{x}^k, \quad \bar{\mb{x}}^k = \mb{J} \mb{x}^k, \notag\\
    & \mb{g}^k = \nabla \tilde{f}(\mb{x}^k), \quad \bar{\mb{g}}^k = \mb{J} \mb{g}^k,\notag \\
    & \mb{g}_a^k = \nabla \tilde{f}(\bar{\mb{x}}^k), \quad \bar{\mb{g}}_a^k = \mb{J} \mb{g}_a^k  = \mb{1}_N \otimes \nabla f(\bar{x}^k), \notag\\
    &\bar{\lambda}_\bL=\lambda_1^\bL, \quad \underline{\lambda}_\bL=\lambda_{N-1}^\bL,\quad \kappa_\bL={\bar{\lambda}_\bL}/{\underline{\lambda}_\bL}\geq 1,\notag\\
    &\bar{\lambda}_\bG=\lambda_2^\bG, \quad \underline{\lambda}_\bG=\lambda_{N}^\bG,\quad \kappa_\bG={\bar{\lambda}_\bG}/{\underline{\lambda}_\bG} \geq 1,\notag\\
    & \mb{s}^k=\mb{q}^k+\mb{g}_a^k.
\end{align}


\subsection{Stationarity Guarantee}
We first analyze the convergence result of $\text{DPP}^2$ under general nonconvexity. Here, we define the following:
\begin{align}\label{def parameter}
    \alpha =& \bar{c}_\alpha \bar{\lambda}_\bG, \quad \theta = \bar{c}_\theta \bar{\lambda}_\bG,\quad \bar{c}_\alpha>1,0<\bar{c}_\theta<1,\gamma>0,\notag\\
    \xi_1 =& \frac{\rho \bar{\lambda}_\bL}{2}(\frac{1}{\kappa_\bG}-\bar{c}_\theta)-(1+\frac{3\bar{M}^2}{4}+\frac{\bar{M}^2\bar{c}_\alpha}{2}),\notag\\
    \xi_2 =& \frac{1}{2}(1+\frac{1}{\gamma})\rho^2\bar{\lambda}_\bL+\frac{\bar{c}_\theta}{4}+\frac{1}{2}\bar{c}_\theta \rho\bar{\lambda}_\bL+\frac{11}{4}\notag\\
    &+\bar{M}^2(\frac{2}{\gamma}+\frac{1}{4}\bar{c}_\theta\rho\bar{\lambda}_\bL+\frac{1}{2}\bar{c}_\theta^2),\notag\\
    \xi_3 =& \frac{\bar{c}_\theta}{2}-\frac{5\gamma}{2}-\frac{1}{\rho^2\underline{\lambda}_\bL^2},\quad \xi_4=\frac{\bar{c}_\theta^2}{4}, \quad \xi_5=\frac{7\bar{c}_\theta^2}{4},\notag\\
    \xi_6 =& \frac{1}{4}, \quad \xi_7=\bar{M}+\frac{21}{2}\bar{M}^2,\notag\\
    \xi_8=&(4+{3\bar{M}^2}+2\bar{M}^2\bar{c}_\alpha)/(2\bar{\lambda}_\bL(\frac{1}{\kappa_\bG}-\bar{c}_\theta)),\notag\\
    \xi_9=& 1/\sqrt{\underline{\lambda}_\bL^2(\frac{\bar{c}_\theta}{2}-\frac{5\gamma}{2})}.
\end{align}
The parameters in \eqref{yk new}--\eqref{qk+1 new} are selected as follows:
\begin{align}
    &\kappa_\bL,\kappa_\bG\geq 1, \,\,1<\bar{c}_\alpha<\frac{\kappa_\bL}{\kappa_\bL-1},\,\, \bar{c}_\theta<\frac{1}{\kappa_\bG}, \,\,\gamma<\frac{\bar{c}_\theta}{5}, \label{range c alpha gamma}\\
    &\rho>\max\{\xi_8, \xi_9\},\label{range rho}\\
    &0<\bar{\lambda}_\bG<\min\{\frac{\xi_1}{\xi_2},\Big(-\xi_4+\sqrt{\xi_4^2+4\xi_3\xi_5}\Big)/(2\xi_5), \frac{\xi_6}{\bar{c}_\alpha\xi_7} \},\label{range bar lambda G}\\
    &\bar{\lambda}_\bG<\alpha < \frac{\xi_6}{\xi_7}, \quad \beta=\frac{\alpha-\bar{\lambda}_\bG}{\underline{\lambda}_\bL}.\label{range alpha}
\end{align}
We will show the well-posedness of the above parameters in the subsequent Lemma~\ref{Lemma vk+1 - vk}, which establishes the dynamics of the sequence
\begin{align}\label{def V} 
    V^k = &\frac{1}{2} \|\mb{x}^k\|^2_\mb{K} + \frac{1}{2}\|\mb{s}^k\|^2_{(\theta \bG+\frac{\bG\bQ}{\rho})\bK}+ \langle \mb{x}^k, \frac{1}{2}\theta\mb{K}\mb{s}^k \rangle+ f(\bar{x}^k) - f^*,
\end{align}
where we define $\bQ=\bL^{\dagger}$.
\begin{lemma}\label{Lemma vk+1 - vk}
    Suppose Assumptions~\ref{assumption smooth}, \ref{assumption optimum} and \ref{assumption GL} hold. Let $\{\mb{x}^k\}$ be the sequence generated by \eqref{yk new}--\eqref{qk+1 new} with the parameters selected by \eqref{range c alpha gamma}--\eqref{range alpha}. Then, for any $k \geq 0$,
	\begin{align}
		&V^{k+1} - V^k - (D_1\|\bw^k\|^2+D_2\|\be^k\|^2)\notag\\
        &\overset{\eqref{tilde Vk+1-tilde Vk leq 0}}{\leq}  - \|\mb{x}^k\|^2_{\bar{\lambda}_{\bG}(\xi_1-\xi_2\bar{\lambda}_{\bG})\mb{K} }- \|\mb{s}^k \|^2_{\bar{\lambda}_\bG^2(\xi_3-\xi_4 \bar{\lambda}_\bG-\xi_5\bar{\lambda}_\bG^2)\bK}-\alpha(\xi_6-\xi_7\alpha)\|\bar{\bg}^k\|^2-\frac{\alpha}{8}\|\bar{\bg}_a^k\|^2<0,\label{tilde Vk+1-tilde Vk leq 0}
	\end{align}
    where 
    \begin{align*}
        D_1 =& \frac{\kappa_\bG}{\bar{\lambda}_\bG}+\frac{2\kappa_\bG^2}{\bar{\lambda}_\bG^2}+2+3\rho^2\bar{\lambda}_\bL^2+\theta \rho^2\bar{\lambda}_\bL^2\bar{\lambda}_\bG+\rho \bar{\lambda}_\bL\bar{\lambda}_\bG+\frac{1}{4}\theta^2\rho\bar{\lambda}_\bL\bar{\lambda}_\bG^2+\frac{2}{\alpha}+\bar{M}+\frac{21}{2}\bar{M}^2,\notag\\
        D_2 =& \beta^2(2+\frac{\kappa_\bG}{\bar{\lambda}_\bG}+\frac{2\kappa_\bG^2}{\bar{\lambda}_\bG^2}),
    \end{align*}
    and $\xi_1,\xi_2,\xi_3,\xi_4,\xi_5,\xi_6,\xi_7$ are given in \eqref{def parameter}.
\end{lemma}
\begin{proof}
    See Appendix~\ref{proof: lemma 1}.
\end{proof}

Next, we present an important result that the sequence $\{D_1\|\bw^k\|^2+D_2\|\be^k\|^2\}$ is summable in expectation.
\begin{lemma}\label{lemma: summable error}
    Suppose the Laplace noises $\be^k$ and $\bw^k$ are independently generated such that: $e_i^k\sim \operatorname{Lap}(\theta_{e,i}^k)$ and $w_i^k\sim \operatorname{Lap}(\theta_{w,i}^k)$, where $\theta_{e,i}^k=r_i^k u_{e,i}$ and $\theta_{w,i}^k=r_i^k u_{w,i}$ with $u_{e,i},u_{w,i}>0$ and $r_i\in(0,1)$. Then,
    \begin{equation}\label{D1w+D2e lemma}
        \mathbb{E}\left[\sum_{k=0}^K(D_1\|\bw^k\|^2+D_2\|\be^k\|^2)\right] \leq (D_1+D_2)\frac{2N\bar{u}^2}{1-\bar{r}^2},
    \end{equation}
    where $\bar{u} \!=\! \max_{i=1,\dots,N}\{u_{e,i},u_{w,i}\}$ and $\bar{r}\!=\!\max_{i=1,\dots,N}\{r_i\}$.
\end{lemma}
\begin{proof}
    See Appendix~\ref{proof of lemma summable error}.
\end{proof}

Based on Lemma~\ref{Lemma vk+1 - vk} and Lemma~\ref{lemma: summable error}, we now analyze the convergence rate of $\text{DPP}^2$ with respect to the optimality gap: 
\begin{equation}\label{def hatWk}
\hat{W}^k = \|\mb{x}^{k}-\bar{\bx}^k\|^2+\frac{1}{N}\|\sum_{i=1}^N \nabla {f}_i(x_i^k)\|^2,
\end{equation} 
where $\bar{\bx}^k$ is defined in \eqref{def variable}. The optimality gap comprises the consensus error and the stationarity error. Based on this measure, we now establish the following theorem.

\begin{theorem}\label{theorem:nonconvex}
	Suppose Assumptions~\ref{assumption smooth}, \ref{assumption optimum} and \ref{assumption GL} hold. Let $\{\mb{x}^k\}$ be the sequence generated by \eqref{yk new}--\eqref{qk+1 new}  with the parameters selected by \eqref{range c alpha gamma}--\eqref{range alpha}. With the initialization $\bq^0=\bd^0=0$, for any $K\in\mathbb{N}$ and $k\in[0,K]$,
    \begin{equation*}
        \frac{\sum _{k=0}^{K}\mathbb{E}[\hat{W}^k]}{K+1}\leq \frac{1}{\zeta_5(K+1)}(\zeta_4\hat{V}^0+\frac{2(D_1+D_2)N\bar{u}^2}{1-\bar{r}^2}),
    \end{equation*}
    where 
    \begin{align}
    \zeta_1 =& 1-c_1+\sqrt{(c_1-1)^2+\theta^2}  \text{ with } c_1=\frac{\bar{\lambda}_\bG}{\kappa_\bG}(\theta +\frac{1}{\rho\bar{\lambda}_\bL}),\notag\\
    \zeta_2 =& 1-c_2+\sqrt{(c_2-1)^2+\theta^2}  \text{ with } c_2=\bar{\lambda}_\bG(\theta +\frac{1}{\rho\underline{\lambda}_\bL}),\notag\\
    \zeta_3 =& \frac{1}{2}-\frac{\zeta_1}{4},\notag\\
    \zeta_4 =& \max\{\frac{1}{2}+\frac{\zeta_2}{4},1\},\notag\\
    \zeta_5 =& \min\{\bar{\lambda}_\bG(\xi_1-\xi_2\bar{\lambda}_\bG),\alpha(\xi_6-\xi_7\alpha)\},\notag\\
    \hat{V}_0=&\|\bx^0-\bar{\bx}^0\|^2+\frac{1}{N}\|\sum_{i=1}^N \nabla {f}_i(x_i^0)\|^2+f(\bar{x}^0)-f^*.\notag
    \end{align}
    The parameters \(\zeta_3\), \(\zeta_4\), and \(\hat{V}^0\) are defined in Theorem~\ref{theorem:nonconvex}; \(D_1\) and \(D_2\) are given in Lemma~\ref{Lemma vk+1 - vk}; \(\bar{u}\) and \(\bar{r}\) are defined in Lemma~\ref{lemma: summable error}; and \(\xi_1\), \(\xi_2\), \(\xi_3\), \(\xi_4\), and \(\xi_5\) are given in \eqref{def parameter}.
\end{theorem}
\begin{proof}
    See Appendix~\ref{proof DPP2 convergence}.
\end{proof}

Theorem~\ref{theorem:nonconvex} indicates that the running average of the optimality gap dissipates and, thus, $\text{DPP}^2$ converges to a stationary solution at a sublinear rate of $\mathcal{O}(1/K)$. The sublinear rate is related to the initialization of the Laplace noise $\bar{u}$ and its decay rate $\bar{r}$, which will be verified by the numerical example in Section~\ref{section:simulation}. Moreover, the rate is of the same order as \cite{hong2017prox,sun2018distributed,sun2019distributed,yi2022sublinear,yi2021linear} for solving smooth nonconvex problems and is better than those in \cite{zhu2024privsgp,DIFF2,wang2023decentralized,Huang2015Diff}. 

\subsection{Global Optimum Guarantee}
Now we provide the convergence analysis of $\text{DPP}^2$ for achieving global optimum under the following condition.
\begin{assumption}\label{assumption:PL}
    The global objective function $f(x)$ satisfies the P-{\L} condition with constant $\nu>0$, i.e.,
	\begin{equation}
		\|\nabla f(x)\|^2 \geq 2\nu (f(x) - f^*), \quad \forall x \in \Re ^d. \label{def pl}
	\end{equation}
\end{assumption}
Note that the P-{\L} condition is milder than the commonly adopted strong convexity \cite{Ding2022TAC,DingTie2018CDC,Huang2015Diff,Yuan2024DistributedNG}. We next present the convergence result of $\text{DPP}^2$ under the P-{\L} condition.

\begin{theorem}\label{theorem:PL}
    Suppose Assumptions~\ref{assumption smooth}, \ref{assumption optimum}, \ref{assumption GL} and \ref{assumption:PL} hold. Let $\{\mb{x}^k\}$ be the sequence generated by \eqref{yk new}--\eqref{qk+1 new} with the parameters selected by \eqref{range c alpha gamma}--\eqref{range alpha}. With the initialization $\bq^0=\bd^0=0$, for any $k\geq 0$,
    \begin{align*}
        &\mathbb{E}[\|\bx^k-\bar{\bx}^k\|^2+f(\bar{x}^k)-f^*]\leq (1-\zeta)^{k}\frac{1}{\zeta_3}\Big(\zeta_4\hat{V}^0+\frac{2(D_1+D_2)N\bar{u}^2}{1-\zeta-\bar{r}^2}\Big),
    \end{align*}
    where 
    \begin{align}
        \label{range zeta 6}\zeta_6 =& \min\{\bar{\lambda}_\bG(\xi_1-\xi_2\bar{\lambda}_\bG),\bar{\lambda}_\bG^2(\xi_3-\xi_4\bar{\lambda}_\bG-\xi_5\bar{\lambda}_\bG^2),\frac{\alpha\nu N}{4},\zeta_4(1-\bar{r}^2)\},\\
        \label{range zeta} 0<& \zeta =\zeta_6/\zeta_4<1.
    \end{align}
    The parameters $\zeta_3,\zeta_4,\hat{V}^0$ are defined in Theorem~\ref{theorem:nonconvex}; $D_1,D_2$ are given in Lemma~\ref{Lemma vk+1 - vk}; $\bar{u},\bar{r}$ are defined in Lemma~\ref{lemma: summable error}; and $\xi_1,\xi_2,\xi_3,\xi_4,\xi_5$ are given in \eqref{def parameter}.
\end{theorem}
\begin{proof}
    See Appendix~\ref{proof theorem PL}.
\end{proof}

Theorem~\ref{theorem:PL} reveals that there exists a constant $\theta\in(0,1)$ such that $\mathbb{E}[\|\bx^k-\bar{\bx}^k\|^2+f(\bar{x}^k)-f^*]=\mathcal{O}(\theta^k)$. This indicates that, as $k\rightarrow \infty$, $x_i^k \,\,\forall i$ reach consensus and $f(\bar{x}^k)$ converges to $f^*$ (the optimal value of the global objective function $f(x)$ defined in Assumption~\ref{assumption optimum}). Hence, $x_i^k \,\,\forall i$ enjoy a linear rate of convergence to the unique global optimum under the P-{\L} condition. This result improves the linear convergence to suboptimality, as is stated in \cite{XIE2025112338} under the P-{\L} condition and in \cite{DingTie2018CDC,Ding2022TAC} under strong convexity. 
Similar to the results in Theorem~\ref{theorem:nonconvex}, the linear rate is governed by the initialization and decay rate of the Laplace noise.

\begin{remark}
    Theorem~\ref{theorem:nonconvex} and Theorem~\ref{theorem:PL} guarantee exact convergence under nonconvex settings, which is superior to the suboptimality guarantees provided by existing methods \cite{zhu2024privsgp,DIFF2,Huang2015Diff,Ding2022TAC,xie_compressed_2023,XIE2025112338,huang_differential_2024}. Moreover, as is stated in Table~\ref{tab:my_label}, under general nonconvex settings, $\text{DPP}^2$ achieves a sublinear convergence rate of $\mathcal{O}(1/K)$, which outperforms both the asymptotic convergence of the method in \cite{wang2023decentralized,wang2024robust,Huang2015Diff} and the $\mathcal{O}(1/\sqrt{K})$ rate of the methods \cite{zhu2024privsgp,DIFF2}. Under the additional P-{\L} condition, $\text{DPP}^2$ further attains a linear convergence rate—while relaxing the strong convexity required by the methods in  \cite{Huang2015Diff,Ding2022TAC,Yuan2024DistributedNG}.
\end{remark}
\section{Differential Privacy}\label{section:privacy}
In this section, we show that the proposed $\text{DPP}^2$ preserves the differential privacy of all the local objective functions. 

Given any objective function $f_{i_0}\in\{f_1,\dots,f_N\}$, note that for any two $\delta$-adjacent function sets $F^{(1)}=\{f_i^{(1)}\}_{i=1}^N$ and $F^{(2)}=\{f_i^{(2)}\}_{i=1}^N$, defined in Definition~\ref{def Adjacency}, the objective functions $f^{(1)}$ and $f^{(2)}$ differ only in $f_{i_0}$, i.e., $f^{(1)}(x)=\sum_{i\neq i_0}f_i+f^{(1)}_{i_0}$ and $f^{(2)}(x)=\sum_{i\neq i_0}f_i+f^{(2)}_{i_0}$. As $\epsilon$-DP measures the indistinguishability of an algorithm’s output when the algorithm is run on two adjacent function sets ($F^{(1)}$ and $F^{(2)}$), we analysis the differential privacy guarantee of $\text{DPP}^2$ in the following theorem.

\begin{theorem}\label{theorem:dp}
	Suppose Assumptions~\ref{assumption smooth}, \ref{assumption optimum} and \ref{assumption GL} hold. Given a time horizon $K>0$ and privacy level $\epsilon_{i_0}>0,i_0\in\mathcal{V}$, $\text{DPP}^2$ preserves the $\epsilon_{i_0}$-differential privacy for any node $i_0$'s objective function if 
    $$\sum_{k=1}^K \sqrt{d}\Big(\frac{1}{\alpha u_{e,{i_0}}}+\frac{1}{u_{w,{i_0}}}\Big)\frac{\alpha \delta}{r_{i_0}^k(1-\alpha \bar{M})}\leq\epsilon_{i_0},$$
    where $u_{e,{i_0}},u_{w,{i_0}}$ and $r_{i_0}$ are defined in Section~\ref{subsection second-tier}.
\end{theorem}
\begin{proof}
	See Appendix~\ref{section:appendix proof of dp}.
\end{proof}

Theorem~\ref{theorem:dp} establishes the differential privacy guarantee of \(\text{DPP}^2\) for protecting local objective functions over a finite time horizon \(K > 0\). It further reveals a key relationship between noise disturbance and privacy protection: specifically, increasing the magnitude of noise disturbance (i.e., larger values of \(u_{e,i_0}\), \(u_{w,i_0}\), and \(r_{i_0}\)) leads to enhanced data privacy, which is reflected by a smaller privacy budget \(\epsilon_{i_0}\).

Notably, Theorem~\ref{theorem:dp} requires no extra assumptions such as bounded gradients (as is stated in \cite{DIFF2,wang2023decentralized,wang2024robust,Huang2015Diff,XIE2025112338}) or identical gradient difference (i.e., $\nabla f_{i_0}^{(1)}(x_a)-\nabla f_{i_0}^{(1)}(x_b)=\nabla f_{i_0}^{(2)}(x_a)-\nabla f_{i_0}^{(2)}(x_b)$, for some $x_a,x_b\in\mathbb{R}^d$, as is stated in \cite{DingTie2018CDC,Ding2022TAC,Yuan2024DistributedNG}). 

Subsequently, we present the selection of parameter $r_{i_0}$ in the following corollary. 

\begin{corollary}\label{corollary:dp parameters}
Suppose Assumptions~\ref{assumption smooth}, \ref{assumption optimum} and \ref{assumption GL} hold. If $u_{e,{i_0}}>\frac{\sqrt{d}\bar{M}}{\epsilon_{i_0}}$, $u_{w,{i_0}}>0$, $\alpha<\min\{\frac{1}{M}, (\epsilon_{i_0}-\frac{\sqrt{d}\bar{M}}{u_{e,{i_0}}})/[\delta(\frac{\sqrt{d}}{u_{w,{i_0}}}+\epsilon_{i_0})]\}$, and $r_{i_0}\in ((\frac{\tilde{c}}{\epsilon_{i_0}})^{\frac{1}{K-1}},1)$ with $\tilde{c}=\sqrt{d}\Big(\frac{1}{\alpha u_{e,{i_0}}}+\frac{1}{u_{w,{i_0}}}\Big)\frac{\alpha \delta}{1-\alpha \bar{M}}>0$, $\text{DPP}^2$ preserves the $\epsilon_{i_0}$-differential privacy for any node $i_0$'s objective function.
\end{corollary}

\textbf{Selection of DP parameters}: To achieve the convergence of $\text{DPP}^2$ with DP, we need to select a feasible stepsize that satisfies both the conditions in Theorem~\ref{theorem:nonconvex} and Corollary~\ref{corollary:dp parameters}. To do this, we let $0<\underline{\alpha}<\bar{\alpha}<(\epsilon_{i_0}-\frac{\sqrt{d}\bar{M}}{u_{e,{i_0}}})/[\delta(\frac{\sqrt{d}}{u_{w,{i_0}}}+\epsilon_{i_0})]$ and select the stepsize $\alpha$ that satisfies $\min\{\bar{\lambda}_\bG, \underline{\alpha}\}<\alpha<\min\{\bar{\alpha},\frac{\xi_6}{\xi_7}\}$ (where $\bar{\lambda}_\bG$ is given in \eqref{range bar lambda G}). Note that $\alpha$ is well-defined since $\underline{\alpha}<\bar{\alpha}$ and $\bar{\lambda}_\bG<\frac{\xi_6}{\xi_7}$ in \eqref{range bar lambda G}. Then, given the required $\epsilon_{i_0}$-DP level, we can select $u_{e,{i_0}}>\frac{\sqrt{d}\bar{M}}{\epsilon_{i_0}}$, $u_{w,{i_0}}>0$. Ultimately, we are able to determine the noise decay rate by $r_{i_0}\in ((\frac{\tilde{c}}{\epsilon_{i_0}})^{\frac{1}{K-1}},1)$.

\section{NUMERICAL EXPERIMENT}\label{section:simulation}
We evaluate the convergence performance of $\text{DPP}^2$ via the distributed binary classification problem with nonconvex regularizers \cite{antoniadis2011penalized}, which adheres to Assumptions~\ref{assumption smooth}--\ref{assumption optimum} and is written as
\begin{equation*}
    f_i(x)=\frac{1}{m}\sum_{s=1}^{m}\log(1+\exp(-y_{is}x^{\mathsf{T}}z_{is}))+\sum_{t=1}^{d}\frac{\lambda\omega([x]_t)^2}{1+\omega([x]_t)^2}.
\end{equation*}
Here, $m$ is the number of data samples of each node. Also, $y_{is}\in \{-1,1\}$ and $z_{is}\in \mathbb{R}^d$ denote the label and the feature for the $s$-th data sample of node $i$, respectively. In the simulation, we set $N=50$, $d=10$, $m=200$ with the regularization parameters $\lambda=0.001$ and $\omega=1$. Additionally, we randomly generate $y_{is}$ and $z_{is}$ for each node $i$, which results in local objective functions with the smoothness parameter $\bar{M}=5.03$. We construct a connected geometric graph with the geometric index $r=0.3$, leading to a network with $50$ nodes and $255$ edges. Experiments are executed on Intel Core i7-8700 CPU @ 3.20GHz, 3192 Mhz, 6 cores with 16GB memory.

We first explore the relationship between the noise decay rate and convergence speed. For all $i\in\mathcal{V}$, we fix $u_{e,i}=u_{w,i}=\bar{u}=1$ and let $r_i=\bar{r}$, where $\bar{r}$ takes on the values 0, 0.5, 0.9, 0.95, 0.97, 0.98, and 0.99. Secondly, to investigate the relationship between the noise initialization and convergence speed, we fix $\bar{r}=0.95$ and set $\bar{u}$ to the values 0, 0.1, 0.3, 0.6, 1, 3, and 5. The parameters of $\text{DPP}^2$ are selected as $\alpha=0.1, \beta=0.05,\rho=10$ and random $\eta^k\in(0,1)$. We measure the optimality by $\|\mb{x}^{k}-\bar{\bx}^k\|^2+\frac{1}{N}\|\sum_{i=1}^N \nabla {f}_i(x_i^k)\|^2$ and show the results in Fig. \ref{fig:q} and Fig. \ref{fig:u}, respectively.

\begin{figure}[h]
    \centering
    \begin{minipage}[t]{0.45\linewidth}
        \centering
        \includegraphics[width=\linewidth]{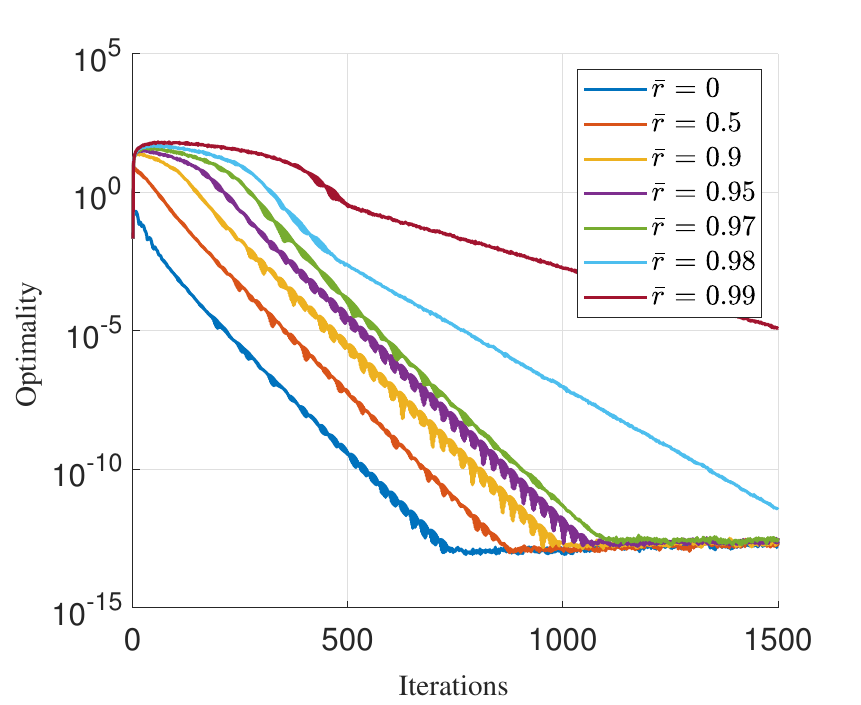}  
        \caption{Convergence performance with respect to $\bar{r}$.}
        \label{fig:q}
    \end{minipage}
    \hfill  
    \begin{minipage}[t]{0.45\linewidth}
        \centering
        \includegraphics[width=\linewidth]{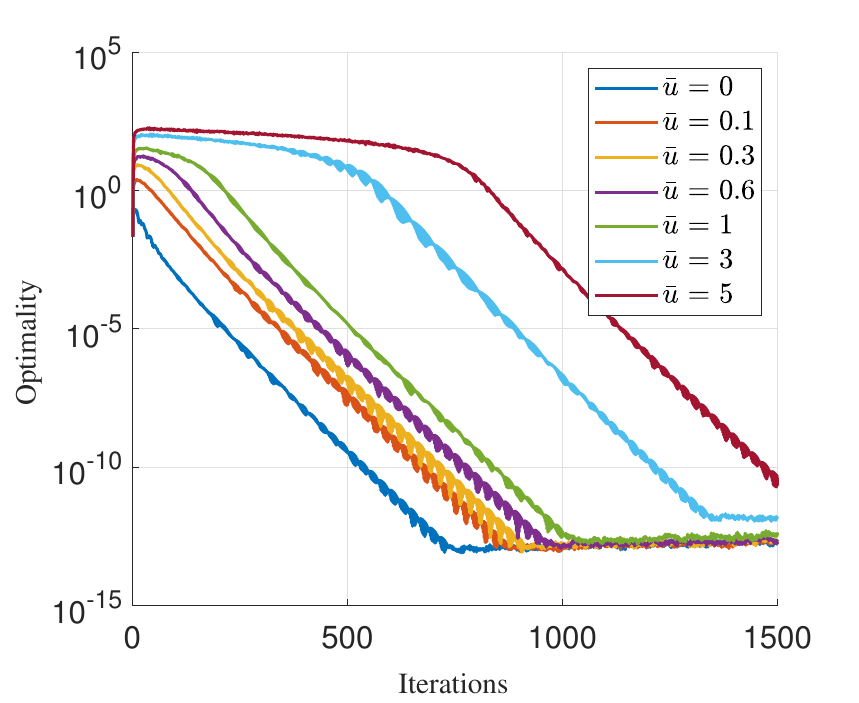}
        \caption{Convergence performance with respect to $\bar{u}$.}
        \label{fig:u}
    \end{minipage}
\end{figure}


We also compare our proposed $\text{DPP}^2$ to algorithms with differential privacy guarantees, including the distributed algorithm via direction and state perturbation (DiaDSP) \cite{Ding2022TAC} and the nonconvex differentially private primal-dual algorithm (PPDC) \cite{XIE2025112338}. Note that DiaDSP can be regarded as a special case of the nonconvex differentially private gradient tracking algorithm (PGTC) \cite{XIE2025112338} without compressed communication. In our experiment, we simulate PPDC without compressed communication and take the noise decay rate $\bar{r}$ to the values 0.1, 0.2, 0.3, 0.4, 0.5, 0.6, 0.7, 0.8, and 0.9. For each specific value of $\bar{r}$, we implement $\text{DPP}^2$, DiaDSP and PPDC for $K=500$ iterations respectively, and then calculate $\|\frac{1}{N}\nabla\tilde{f}(\mathbf{x}^K)\|^2$. We set the algorithm parameters of these algorithms to reach the same privacy budgets (i.e., privacy levels). Specifically, in $\text{DPP}^2$, we set $\rho=10,\alpha=0.0994,\beta=0.05$, $u_{w,i} = 0.994$, $u_{e,i} = 1$ and random $\eta^k\in(0,1)$; In DiaDSP, we set $\alpha=0.01$, $u_{x,i} = 2$ and $u_{y,i} =5$; In PPDC, we let $\gamma = 65, \omega = 5, \eta= 0.01$, $u_{x,i} = 1$ and $u_{v,i} = 1$. We present the numerical result in Fig. \ref{fig:eps}.

\begin{figure}[h]
    \centering
    \includegraphics[width=0.5\linewidth]{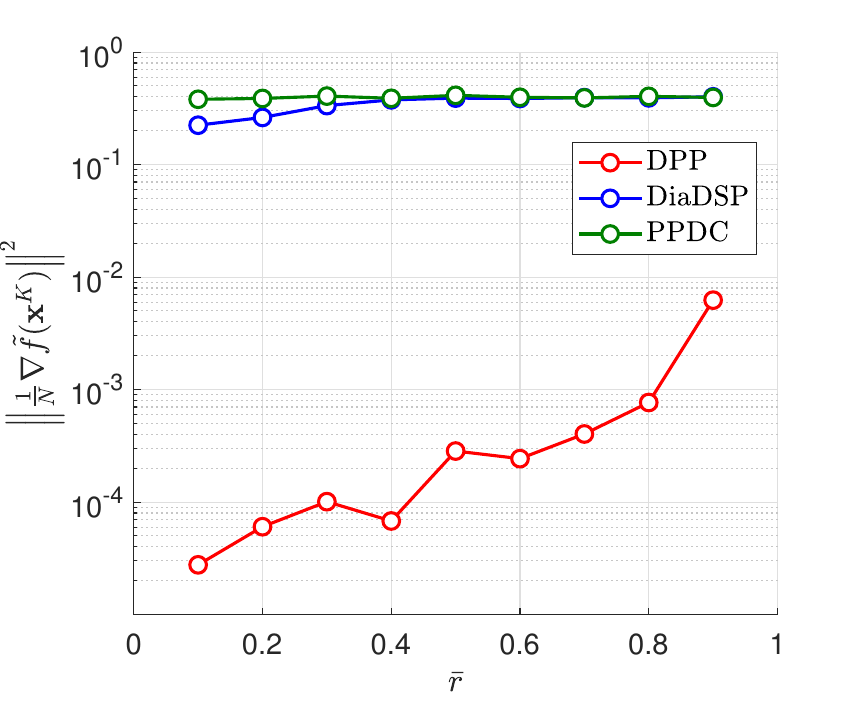}
    \caption{Convergence performance with respect to $\bar{r}$.}
    \label{fig:eps}
\end{figure}

The simulation results in Fig. \ref{fig:q} and Fig. \ref{fig:u} verify the exact convergence properties of $\text{DPP}^2$ with double privacy protection. Moreover, the results also reveal that the convergence performance is better with smaller noise decay rates ($\bar{r}$) and initialization values ($\bar{u}$), which aligns with the theoretical guarantees in Theorem~\ref{theorem:nonconvex} and Theorem~\ref{theorem:PL}. Furthermore, the comparative study presented in Fig. \ref{fig:eps} illustrates a trade-off between differential privacy levels and convergence speed: higher privacy levels (reflected by larger $\bar{r}$) inevitably lead to slower convergence. Notably, under identical level of DP, $\text{DPP}^2$ demonstrates superior convergence performance compared to DiaDSP and PPDC.

\section{CONCLUSIONS}\label{section conclusion}
We have proposed a decentralized proximal primal-dual algorithm with double protection of privacy, referred to as $\text{DPP}^2$, for solving nonconvex, smooth optimization problems, in which a novel two-tier privacy protection strategy is designed. Specifically, the privacy protection strategy adopts decaying Laplace noise to achieve both exact convergence and differential privacy. It also lets the agents transmit mixed variables to protect local decisions and gradients from being eavesdropped even when the Laplace noise becomes tiny. By leveraging decaying Laplace noise, $\text{DPP}^2$ exhibits sublinear convergence to stationary solutions and linear convergence to the global optimum under the P-{\L} condition. These convergence results outperform the alternative methods in terms of convergence speed and solution accuracy. The numerical results demonstrate that, compared with other state-of-the-art differential privacy algorithms, $\text{DPP}^2$ not only attains exact convergence but also enjoys faster convergence while maintaining the same level of differential privacy.
			
			\appendix
			
			\subsection{Proof of Lemma~\ref{Lemma vk+1 - vk}} \label{proof: lemma 1}
\begin{proof}
    (\romannumeral1) First, we show that the parameters in \eqref{def parameter} and the parameter selections in \eqref{range c alpha gamma}--\eqref{range alpha} are well-defined.

    From $\bar{c}_\theta<\frac{1}{\kappa_\bG}$ in \eqref{range c alpha gamma} and $2(1+\frac{3\bar{M}^2}{4}+\frac{\bar{M}^2\bar{c}_\alpha}{2})/(\rho\bar{\lambda}_\bL(\frac{1}{\kappa_\bG}-\bar{c}_\theta))$ in \eqref{range rho}, we have $\xi_1>0$. From $\gamma<\frac{\bar{c}_\theta}{5}$ in \eqref{range c alpha gamma} and $1/\sqrt{\underline{\lambda}_\bL^2(\frac{\bar{c}_\theta}{2}-\frac{5\gamma}{2})}$ in \eqref{range rho}, we obtain $\xi_3>0$. Since $\bar{c}_\alpha>1$ in \eqref{range c alpha gamma}, $\bar{\lambda}_\bG<\frac{\xi_6}{\bar{c}_\alpha\xi_7}$ in \eqref{range bar lambda G}, we ensure the well-posedness of $\alpha$ in \eqref{range alpha}. Moreover, due to $1<\bar{c}_\alpha<\frac{\kappa_\bL}{\kappa_\bL-1}$ in \eqref{range c alpha gamma}, we derive 
    \begin{align*}
        \underline{\lambda}_\bG\overset{\eqref{def G}}{=}&(1-\kappa_\bL)\alpha+\kappa_\bL\bar{\lambda}_\bG\overset{\eqref{def parameter}}{=}(\kappa_\bL-(\kappa_\bL-1)\bar{c}_\alpha)\bar{\lambda}_\bG >0,
    \end{align*}
    and thus $\bG\succ \mathbf{O}_{Nd}$.

    (\romannumeral2) Subsequently, we establish some results for the weight matrices. From \eqref{def G}, there exists an orthogonal matrix $\tilde{\bR}\in\mathbb{R}^{N\times N}$ with the first column $\tilde{\bR}_{*1}=\frac{1}{\sqrt{N}}\mathbf{1}_{N}$ such that $\bL=\tilde{\bR}\Lambda_\bL\tilde{\bR}^\mathsf{T}\otimes \bI_d$, where $\Lambda_\bL=\operatorname{diag}([0,\lambda_{N-1}^\bL,\dots,\lambda_{1}^\bL])$ with $0<\lambda_{N-1}^\bL<\cdots<\lambda_{1}^\bL$. Similarly, $\bG=\tilde{\bR}\Lambda_\bG\tilde{\bR}^\mathsf{T}\otimes \bI_d$, where $\Lambda_\bG=\operatorname{diag}([\alpha,\lambda_{2}^\bG,\dots,\lambda_{N}^\bG])$ with $0<\lambda_{N}^\bG<\cdots<\lambda_{2}^\bG<\alpha$. Also, due to Assumption~\ref{assumption GL}, $\operatorname{Null}(\bL)=\operatorname{Null}(\bK)=\mathcal{S}$, and thus
\begin{gather}
	\bK\bL=\bL\bK=\bL, \quad \bJ\bL=\bJ\bK=\mathbf{O}_{Nd},\label{JL=JK=0}\\
	\bK\bK=\bK, \quad \bJ\bJ=\bJ.\label{KK=K JJ=J}
\end{gather}
Additionally, from the definition $\bQ=\bL^{\dagger}$, we have
\begin{equation}\label{QL=LQ=K}
    \bQ\bL=\bL\bQ=\bK,
\end{equation}
and the matrices $\bQ,\bJ,\bK,\bL$ are commutative with each other in matrix multiplication.

Subsequently, we establish some results based on the parameter selections in \eqref{range c alpha gamma}--\eqref{range alpha}. Eq. \eqref{range c alpha gamma} and \eqref{range rho} imply
\begin{equation}\label{rho underlinelambda >1}
    \theta <1,\quad \rho \underline{\lambda}_\bL >1.
\end{equation}
From \eqref{range bar lambda G}, we have 
\begin{gather}\label{rho L G}
	\mathbf{O}\preceq\rho\bL\bG\preceq \bK\preceq\bI.\\
	\label{lambda G<1}
	\bar{\lambda}_\bG<1,\quad \mathbf{O}\preceq\theta\rho\bL\preceq\bK.
\end{gather}
It follows from \eqref{def G} and $\mb{J}\bL=\mathbf{O}_{Nd}$ in \eqref{JL=JK=0} that
	\begin{align}
		\mb{J}\mb{G} = & \mb{J}(\alpha \mathbf{I}_{Nd}-\beta \bL)=\alpha \mb{J}. \label{JGk}
	\end{align}
From \eqref{xk+1 final}, \eqref{JL=JK=0} and \eqref{JGk}, we have 
\begin{align}
	\bar{\mb{x}}^{k+1} \!-\! \bar{\mb{x}}^k &= -\mb{J}\big( \mb{G}(\mb{g}^k \!+\!\mb{q}^k\!+\!\rho \bL(\bx^k\!+\!\bw^k))\!+\!\bw^k\!+\!\beta\bL\be^k \big)\notag\\
    &= - (\alpha\bar{\mb{g}}^k+\bJ\bw^k). \label{bar xk+1 - bar xk}
\end{align}
Due to \eqref{smooth tilde f}, we have
\begin{align}
	\|\mb{g}_a^k - \mb{g}^k\|^2 \leq \bar{M}^2\|\bar{\mb{x}}^k - \mb{x}^k\|^2 = \bar{M}^2 \|\mb{x}^k\|^2_\mb{K}. \label{g0k-gk}
\end{align}
Then, since $\lambda_1^{\mb{J}} = 1$, we have
\begin{align}
	\|\bar{\mb{g}}_a^k - \bar{\mb{g}}^k\|^2 = & \|\mb{J}(\mb{g}_a^k - \mb{g}^k)\|^2 \overset{\eqref{g0k-gk}}{\leq} \bar{M}^2 \|\mb{x}^k\|^2_\mb{K}. \label{bar g0k- bar gk}
\end{align}
From \eqref{smooth tilde f} and \eqref{bar xk+1 - bar xk}, we have
\begin{align}
	\|\mb{g}_a^{k+1} - \mb{g}_a^k\|^2 \leq& \bar{M}^2 \|\bar{\mb{x}}^{k+1} - \bar{\mb{x}}^{k}\|^2 =  \bar{M}^2\|\alpha\bar{\mb{g}}^k-\bJ\bw^k\|^2\notag\\
    \leq& 2\bar{M}^2(\alpha^2\|\bar{\bg}^k\|^2+\|\bw^k\|^2). \label{g0k+1-g0k}
\end{align}

    (\romannumeral3) We then bound each term of the definition of $V^{k+1}$ (given by \eqref{def V}).

	For the first term of the definition of $V^{k+1}$,

	\begin{align}\label{Vk term1}
		&\frac{1}{2}\|\bx^{k+1}\|^2_\bK \notag\\
		\overset{\eqref{xk+1 new}}{=}&\frac{1}{2}\|\bx^k-\bG\big(\bg^k-\bg_a^k+\bq^k+\bg_a^k+\rho \bL(\bx^k+\bw^k) \big)+\bw^k+\beta \bL\be^k\|^2_\bK \notag\\
		=& \frac{1}{2}\|(\bI-\rho\bL\bG)\bx^k+(\bI-\rho\bL\bG)\bw^k+\beta\bL\be^k\|^2_\bK\notag\\
		&-\langle (\bI-\rho\bL\bG)\bx^k+(\bI-\rho\bL\bG)\bw^k+\beta\bL\be^k,\bG\bK(\bg^k-\bg_a^k+\bq^k+\bg_a^k) \rangle\notag\\
        &+\frac{1}{2}\|\bg^k-\bg_a^k+\bq^k+\bg_a^k\|^2_{\bG^2\bK}\notag\\
		\leq & \frac{1}{2}\|\bx^k\|_\bK^2-\frac{1}{2}\|\bx^k\|^2_{2\rho\bL\bG-\rho^2\bL^2\bG^2}+\frac{1}{2}\|(\bI-\rho\bL\bG)\bx^k\|^2_{\bG\bK}+\frac{1}{2}\|(\bI-\rho\bL\bG)\bw^k+\beta\bL\be^k\|^2_{(\bG^{-1}+\bI)\bK}\notag\\
        &+\frac{1}{2}\|(\bI-\rho\bL\bG)\bx^k\|^2_{\bG\bK}+\frac{1}{2}\|\bg^k-\bg_a^k\|^2_{\bG\bK}-\langle \bx^k,\bG\bK\bs^k\rangle+\frac{1}{2\gamma}\|\bx^k\|^2_{ \rho^2\bL^2\bG^2}+\frac{1}{2}\|\bs^k\|^2_{\gamma\bG^2\bK}\notag\\
        &+\frac{1}{2}\|(\bI-\rho\bL\bG)\bw^k+\beta\bL\be^k\|^2_\bK+\frac{1}{\gamma}\|\bg^k-\bg_a^k\|^2_{\bG^2\bK}+\|\bs^k\|^2_{\gamma\bG^2\bK}\notag\\
        & +\frac{1}{\gamma}\|\bg^k-\bg_a^k\|^2_{\bG^2\bK}+\|\bs^k\|^2_{\gamma\bG^2\bK}\notag\\
		\overset{\eqref{rho L G}}{\leq}&\frac{1}{2}\|\bx^k\|^2_\bK-\|\bx^k\|^2_{\rho\bL\bG-(\frac{1}{2}(1+\frac{1}{\gamma})\rho^2\bL^2\bG+\bG\bK)}-\langle\bx^k,\bG\bK\bs^k\rangle+\|\bg^k-\bg_a^k\|^2_{\frac{1}{2}\bG\bK+\frac{2}{\gamma}\bG^2\bK}\notag\\
		&+\frac{5}{2}\|\bs^k\|^2_{\gamma \bG^2\bK}+\|\bw^k\|^2_{(\bG^{-1}+2\bI)\bK}+\|\beta\bL\be^k\|^2_{(\bG^{-1}+2\bI)\bK}\notag\\
		\overset{\eqref{g0k-gk}}{\leq} & \frac{1}{2}\|\bx^k\|^2_\bK -\langle \bx^k,\bG\bK\bs^k\rangle+\frac{5}{2}\|\bs^k\|^2_{\gamma\bG^2\bK}- \|\bx^k\|^2_{\rho\bL\bG-(\frac{1}{2}(1+\frac{1}{\gamma})\rho^2\bL^2\bG^2+\bG\bK+(\frac{1}{2}\bar{\lambda}_\bG+\frac{2\bar{\lambda}_\bG^2}{\gamma})\bar{M}^2\bK)}\notag\\
        &+(\frac{1}{\underline{\lambda}_{\bG}}+2)(\|\bw^k\|^2+\beta^2\|\be^k\|^2).
	\end{align}
    
	For the second term of the definition of $V^{k+1}$,
	\begin{align}\label{Vk term2}
		&\frac{1}{2}\|\bs^{k+1}\|^2_{\theta \bG\bK+\frac{\bQ\bG}{\rho}}=\frac{1}{2}\|\bq^{k+1}+\bg_a^{k+1}\|_{\theta \bG\bK+\frac{\bQ\bG}{\rho}}^2\notag\\
		\overset{\eqref{qk+1 new}}{=}&\frac{1}{2}\|\bq^k+\bg_a^k+\rho\bL(\bx^k+\bw^k)+\bg_a^{k+1}-\bg_a^k\|_{\theta \bG\bK+\frac{\bQ\bG}{\rho}}^2\notag\\
		=&\frac{1}{2}\|\bs^k\|_{\theta \bG\bK+\frac{\bQ\bG}{\rho}}^2+\langle (\theta \bG\bK+\frac{\bQ\bG}{\rho})\bs^k,\rho \bL\bx^k\rangle+\langle(\theta \bG\bK+\frac{\bQ\bG}{\rho})\bs^k,\rho\bL\bw^k+\bg^{k+1}-\bg_a^k\rangle\notag\\
		&+\frac{1}{2}\|\rho\bL \bx^k+\rho\bL\bw^k+\bg_a^{k+1}-\bg_a^k\|_{\theta \bG\bK+\frac{\bQ\bG}{\rho}}^2\notag\\
		\overset{\eqref{QL=LQ=K}}{\leq} &\frac{1}{2}\|\bs^k\|_{\theta \bG\bK+\frac{\bQ\bG}{\rho}}^2 + \langle (\theta\rho \bL\bG+\bG\bK)\bx^k,\bs^k\rangle +\|\bs^k\|^2_{\theta^2\bG^2\bK+\frac{\bQ^2\bG^2}{\rho^2}}+\|\bw^k\|^2_{\rho^2\bL^2}+\|\bg_a^{k+1}-\bg_a^k\|^2\notag\\
		&+\frac{1}{2}\|\bx^k\|^2_{\theta\rho^2\bL^2\bG+\rho\bL\bG}\!+\!\frac{1}{2}\|\rho \bL\bw^k+\bg_a^{k+1}\!-\!\bg_a^k\|_{\theta \bG\bK+\frac{\bQ\bG}{\rho}}^2\notag\\
        &+\langle (\theta \rho \bL\bG+\bG\bK)\bx^k,\rho \bL\bw^k+\bg_a^{k+1}-\bg_a^k\rangle\notag\\
		\leq & \frac{1}{2}\|\bs^k\|_{\theta \bG\bK+\frac{\bQ\bG}{\rho}}^2 + \langle (\theta\rho \bL\bG+\bG\bK)\bx^k,\bs^k\rangle+\|\bs^k\|^2_{\theta^2\bG^2\bK+\frac{\bQ^2\bG^2}{\rho^2}}+\|\bw^k\|^2_{\rho^2\bL^2}+\|\bg_a^{k+1}-\bg_a^k\|^2\notag\\
		&+\frac{1}{2}\|\bx^k\|^2_{\theta\rho^2\bL^2\bG+\rho\bL\bG}+\|\bx^k\|^2_{\theta^2\rho^2\bL^2\bG^2+\bG^2\bK}+\|\bw^k\|^2_{\rho^2\bL^2}+\|\bg_a^{k+1}-\bg_a^k\|^2\notag\\
        &+\|\bw^k\|^2_{\theta\rho^2\bL^2\bG+\rho\bL\bG}+\|\bg_a^{k+1}-\bg_a^k\|^2_{\theta\bG\bK+\frac{\bQ\bG}{\rho}}\notag\\
		\overset{\eqref{rho underlinelambda >1}}{\leq} & \frac{1}{2}\|\bs^k\|_{\theta \bG\bK+\frac{\bQ\bG}{\rho}}^2 + \langle (\theta\rho \bL\bG+\bG\bK)\bx^k,\bs^k\rangle +\|\bx^k\|^2_{\frac{1}{2}\theta\rho^2\bL^2\bG+\frac{1}{2}\rho \bL\bG+\theta^2\rho^2\bL^2\bG^2+\bG^2\bK}\notag\\
        &+\|\bs^k\|^2_{\theta^2\bG^2\bK+\frac{\bQ^2\bG^2}{\rho^2}}+\|\bw^k\|^2_{2\rho^2\bL^2+\theta\rho^2\bL^2\bG+\rho\bL\bG}+4\|\bg_a^{k+1}-\bg_a^k\|^2\notag\\
		\underset{\eqref{lambda G<1}}{\overset{\eqref{g0k+1-g0k}}{\leq}}&\frac{1}{2}\|\bs^k\|_{\theta \bG\bK+\frac{\bQ\bG}{\rho}}^2 + \langle (\theta\rho \bL\bG+\bG\bK)\bx^k,\bs^k\rangle +\|\bx^k\|^2_{\frac{1}{2}\theta\rho^2\bL^2\bG+\frac{1}{2}\rho \bL\bG+2\bG^2\bK}\notag\\               &+\Big(2\rho^2\bar{\lambda}_\bL^2+\theta\rho^2\bar{\lambda}_\bL^2\bar{\lambda}_\bG+\rho\bar{\lambda}_\bL\bar{\lambda}_\bG+8\bar{M}^2\Big)\|\bw^k\|^2+\|\bs^k\|^2_{\theta^2\bG^2\bK+\frac{\bQ^2\bG^2}{\rho^2}}+8\bar{M}^2\alpha^2\|\bar{\bg}^k\|^2.
	\end{align}
    
	For the third term of the definition of $V^{k+1}$,
	\begin{align}\label{Vk term3}
		&\langle \bx^{k+1},\theta \bK \bs^{k+1}\rangle\notag\\
		=&\langle \bx^k-\bG(\bq^k+\bg_a^k+\bg^k-\bg_a^k+\rho \bL (\bx^k+\bw^k))+\bw^k+\beta\bL\be^k,\notag\\
        &\theta\bK(\bq^k+\bg_a^k+\rho\bL(\bx^k+\bw^k)+\bg_a^{k+1}-\bg_a^k)\rangle\notag\\
		=&\langle \bx^k,\theta\bK\bs^k\rangle+\|\bx^k\|^2_{\theta\rho\bL}+\langle \bx^k,\theta\bK(\rho\bL\bw^k+\bg_a^{k+1}-\bg_a^k)\rangle-\|\bs^k\|^2_{\theta\bG\bK}\notag\\
		&-\langle\theta\rho\bL\bG\bs^k,\bx^k\rangle-\langle\theta\bG\bK\bs^k,\rho\bL\bw^k+\bg_a^{k+1}-\bg_a^k\rangle\notag\\
        &-\langle \theta\bG\bK(\bg^k-\bg_a^k),\bs^k+\rho\bL(\bx^k+\bw^k)+\bg_a^{k+1}-\bg_a^k\rangle \notag\\
		&-\langle\theta\rho\bL\bG\bx^k,\bs^k+\rho\bL\bw^k+\bg_a^{k+1}-\bg_a^k\rangle-\|\bx^k\|^2_{\theta\rho^2\bL^2\bG}\notag\\
		&+\langle\beta\bL\be^k+(\bI-\rho\bL\bG)\bw^k,\theta\bK(\bs^k+\rho\bL(\bx^k+\bw^k)+\bg_a^{k+1}-\bg_a^k)\rangle\notag\\
		\leq & \langle \bx^k,\theta\bK\bs^k\rangle+\|\bx^k\|^2_{\theta\rho\bL}+\|\bx^k\|^2_{\theta^2\bK}+\frac{1}{2}\|\bw^k\|^2_{\rho^2\bL^2}+\frac{1}{2}\|\bg_a^{k+1}-\bg_a^k\|^2-\|\bs^k\|_{\theta\bG\bK}\notag\\
        &-\langle\theta\rho \bL\bG\bs^k,\bx^k\rangle+\|\bs^k\|^2_{\theta^2\bG^2\bK}+\frac{1}{2}\|\bw^k\|^2_{\rho^2\bL^2}+\frac{1}{2}\|\bg_a^{k+1}-\bg_a^k\|^2+\frac{1}{2}\|\bg^k-\bg_a^k\|^2_{\bG\bK}\notag\\
		&+\frac{1}{2}\|\bs^k\|^2_{\theta^2\bG\bK}+\frac{1}{2}\|\bg^k-\bg_a^k\|^2_{\theta\rho\bL\bG}+\frac{1}{2}\|\bx^k\|^2_{\theta\rho\bL\bG}+\|\bg^k-\bg_a^k\|^2_{\theta^2\bG^2\bK}+\frac{1}{2}\|\bw^k\|^2_{\rho^2\bL^2}\notag\\
		&+\frac{1}{2}\|\bg_a^{k+1}-\bg_a^k\|^2-\langle\theta\rho \bL\bG\bx^k,\bs^k\rangle+\|\bx^k\|^2_{\theta^2\rho^2\bL^2\bG^2}+\frac{1}{2}\|\bw^k\|^2_{\rho^2\bL^2}+\frac{1}{2}\|\bg_a^{k+1}-\bg_a^k\|^2\notag\\
		&-\|\bx^k\|^2_{\theta\rho^2\bL^2\bG}+2\|\beta\bL\be^k+(\bI-\rho\bL\bG)\bw^k\|^2_{\bG^{-2}\bK}+\frac{1}{2}\|\bs^k\|^2_{\theta^2\bG^2\bK}+\frac{1}{2}\|\bx^k\|^2_{\theta^2\rho\bL\bG^2}\notag\\
        &+\frac{1}{2}\|\bw^k\|^2_{\theta^2\rho\bL\bG^2}+\frac{1}{2}\|\bg_a^{k+1}-\bg_a^k\|^2_{\theta^2\bG^2\bK}\notag\\
		\underset{\eqref{lambda G<1}}{\overset{\eqref{rho underlinelambda >1}\eqref{rho L G}}{\leq}} & \langle \bx^k,\theta\bK\bs^k\rangle -2\langle \theta\rho\bL\bG\bx^k,\bs^k\rangle-\|\bs^k\|^2_{\theta\bG\bK-(\frac{3}{2}\theta^2\bG^2\bK+\frac{1}{2}\theta^2\bG\bK)}\notag\\
        &+\|\bx^k\|^2_{-\theta\rho^2\bL^2\bG+\theta\rho\bL+\frac{1}{2}\theta^2\bK+\theta\rho\bL\bG+\frac{3}{2}\bG^2\bK}+\|\bg^k-\bg_a^k\|^2_{\frac{1}{2}\bG\bK+\frac{1}{2}\theta\rho\bL\bG+\theta^2\bG^2\bK}\notag\\
        &+\frac{5}{2}\|\bg_a^{k+1}-\bg_a^k\|^2+\|\bw^k\|^2_{2\rho^2\bL^2+\frac{1}{2}\theta^2\rho\bL\bG^2}+\frac{4}{\underline{\lambda}_\bG^2}(\beta^2\|\be^k\|^2+\|\bw^k\|^2)\notag\\
		\underset{\eqref{bar g0k- bar gk}}{\overset{\eqref{g0k+1-g0k}}{\leq}} & \langle \bx^k,\theta\bK\bs^k\rangle -2\langle \theta\rho\bL\bG\bx^k,\bs^k\rangle-\|\bs^k\|^2_{\theta\bG\bK-(\frac{3}{2}\theta^2\bG^2\bK+\frac{1}{2}\theta^2\bG\bK)}\notag\\
        &+\|\bx^k\|^2_{-\theta\rho^2\bL^2\bG+\theta\rho\bL+\frac{1}{2}\theta^2\bK+\theta\rho\bL\bG+\frac{3}{2}\bG^2\bK}+\bar{\lambda}_\bG(\frac{1}{2}+\frac{1}{2}\theta\rho\bar{\lambda}_{\bL}+\theta^2\bar{\lambda}_{\bG})\bar{M}^2\|\bx^k\|^2_{\bK}\notag\\
        &+\Big(5\bar{M}^2+2\rho^2\bar{\lambda}_\bL^2+\frac{1}{2}\theta^2\rho\bar{\lambda}_\bL\bar{\lambda}_\bG^2+\frac{4}{\underline{\lambda}_\bG^2}\Big)\|\bw^k\|^2+5\bar{M}^2\alpha^2\|\bar{\bg}^k\|^2+\frac{4\beta^2}{\underline{\lambda}_\bG^2}\|\be^k\|^2.
	\end{align}
	For the last term of the definition of $V^{k+1}$,
	\begin{align}\label{Vk term4}
		&f(\bar{x}^{k+1})-f^*=\tilde{f}(\bar{\bx}^{k+1})-f^*\notag\\
		=&\tilde{f}(\bar{\bx}^{k})-f^*+\tilde{f}(\bar{\bx}^{k+1})-\tilde{f}(\bar{\bx}^{k})\notag\\
		\leq & \tilde{f}(\bar{\bx}^{k})-f^*+\langle\bar{\bx}^{k+1}-\bar{\bx}^k,\nabla \tilde{f}(\bar{\bx}^k)\rangle+\frac{\bar{M}}{2}\|\bar{\bx}^{k+1}-\bar{\bx}^k\|^2\notag\\
		\overset{\eqref{bar xk+1 - bar xk}}{\leq} & \tilde{f}(\bar{\bx}^{k})-f^* - \langle\bJ(\alpha\bg^k+\bw^k),\bg_a^k\rangle+\frac{\bar{M}}{2}\|\alpha\bar{\mb{g}}^k+\bJ\bw^k\|^2\notag\\
		\overset{\eqref{KK=K JJ=J}}{\leq}& \tilde{f}(\bar{\bx}^{k})-f^* - \langle\alpha\bar{\bg}^k+\bJ\bw^k,\bar{\bg}_a^k\rangle+\frac{\bar{M}}{2}\|\alpha\bar{\mb{g}}^k+\bJ\bw^k\|^2\notag\\
		=&\tilde{f}(\bar{\bx}^{k})-f^* -\frac{\alpha}{2}\langle\bar{\bg}^k,\bar{\bg}^k+\bar{\bg}_a^k-\bar{\bg}^k\rangle-\frac{\alpha}{2}\langle\bar{\bg}^k-\bar{\bg}_a^k\!+\!\bar{\bg}_a^k,\bar{\bg}_a^k\rangle\notag\\
		&-\langle\bJ\bw^k,\bar{\bg}_a^k\rangle+\alpha^2\bar{M}\|\bar{g}^k\|^2+\bar{M}\|\bw^k\|^2\notag\\
		\leq&\tilde{f}(\bar{\bx}^{k})-f^*-\frac{\alpha}{4}\|\bar{\bg}^k\|^2+\frac{\alpha}{4}\|\bar{\bg}_a^k-\bar{\bg}^k\|^2-\frac{\alpha}{4}\|\bar{\bg}_a^k\|^2\notag\\
		&+\!\frac{\alpha}{4}\|\bar{\bg}_a^k-\bar{\bg}^k\|^2\!+\!\frac{\alpha}{8}\|\bar{\bg}_a^k\|^2\!+\!(\frac{2}{\alpha}\!+\!\bar{M})\|\bw^k\|^2\!+\!\alpha^2\bar{M}\|\bar{\bg}^k\|^2\notag\\
		=&\tilde{f}(\bar{\bx}^{k})-f^*-\alpha(\frac{1}{4}-\bar{M}\alpha)\|\bar{\bg}^k\|^2+\frac{\alpha}{2}\|\bar{\bg}_a^k-\bar{\bg}^k\|^2-\frac{\alpha}{8}\|\bar{\bg}_a^k\|^2+(\frac{2}{\alpha}+\bar{M})\|\bw^k\|^2\notag\\
		\overset{\eqref{bar g0k- bar gk}}{\leq}&\tilde{f}(\bar{\bx}^{k})-f^*-\alpha(\frac{1}{4}-\bar{M}\alpha)\|\bar{\bg}^k\|^2+\frac{\alpha\bar{M}^2}{2}\|\bx^k\|^2_{\bK}-\frac{\alpha}{8}\|\bar{\bg}_a^k\|^2+(\frac{2}{\alpha}+\bar{M})\|\bw^k\|^2.
	\end{align}
	Combining \eqref{Vk term1}--\eqref{Vk term4} yields \eqref{tilde Vk+1-tilde Vk leq 0}.

(\romannumeral4) Finally, we illustrate how the sequence in \eqref{tilde Vk+1-tilde Vk leq 0} descends along iterations based on the well-defined parameters in \eqref{range c alpha gamma}--\eqref{range alpha}.

From $0 < \bar{\lambda}_{\mb{G}} < \min\{\frac{\xi_1}{\xi_2},(-\xi_4+\sqrt{\xi_4^2+4\xi_3\xi_5})/(2\xi_5)\}$ in \eqref{range bar lambda G}, we have
\begin{align}
	\label{xi1-xi2} & \bar{\lambda}_{\bG}(\xi_1-\xi_2\bar{\lambda}_{\bG}) > 0 , \\
	\label{xi3-xi5} & \bar{\lambda}_{\bG}^2(\xi_3-\xi_4\bar{\lambda}_{\bG}-\xi_5\bar{\lambda}_{\bG}^2)> 0.
\end{align}
Since $\bar{\lambda}_\bG<\alpha < \frac{\xi_6}{\xi_7}$ in \eqref{range alpha}, we obtain
\begin{align}
	\alpha(\xi_6-\xi_7\alpha)>0.\label{xi6-xi7}
\end{align}
Ultimately, combining \eqref{xi1-xi2}--\eqref{xi6-xi7} yields $V^{k+1} - V^k - (D_1\|\bw^k\|^2+D_2\|\be^k\|^2)\leq 0$, and thus \eqref{tilde Vk+1-tilde Vk leq 0} holds.
\end{proof}

\subsection{Proof of Lemma~\ref{lemma: summable error}} \label{proof of lemma summable error}
According to the Laplace noises $\bw^k$ and $\be^k$, we have
\begin{align*}
    &\mathbb{E}[\sum_{k=0}^K(D_1\|\bw^k\|^2+D_2\|\be^k\|^2)]\notag\\
    =&\sum_{k=0}^K\mathbb{E}[D_1\|\bw^k\|^2]+\sum_{k=0}^K\mathbb{E}[D_2\|\be^k\|^2]\notag\\
    \leq&(D_1+D_2)\sum_{k=0}^\infty(2N\bar{r}^{2k}\bar{u}^2)\notag\\
    =&(D_1+D_2)\frac{2N\bar{u}^2}{1-\bar{r}^2},
\end{align*}
where $\bar{u} = \max_{i=1,\dots,N}\{u_{e,i},u_{w,i}\}$ and $\bar{r}=\max_{i=1,\dots,N}\{r_i\}$. 

\subsection{Proof of Theorem~\ref{theorem:nonconvex}}\label{proof DPP2 convergence}
First, we define
\begin{equation}\label{def hat V}
    \hat{V}^k = \|\bx^k\|^2_\bK+\|\bs^k\|^2_\bK+f(\bar{x}^k)-f^*.
\end{equation}
From the definition in \eqref{def V}, we obtain
\begin{align}\label{Vk geq hat Vk}
	V^k\geq &\frac{1}{2}\|\bx^k\|^2_\bK+\frac{1}{2}\underline{\lambda}_{\bG}(\theta+\frac{1}{\rho\bar{\lambda}_\bL})\|\bs^k\|^2_\bK-\frac{\zeta_1}{4}\|\bx^k\|^2_\bK-\frac{\theta}{4\zeta_1}\|\bs^k\|^2_{\bK}+f(\bar{x}^k)-f^*\notag\\
	=&(\frac{1}{2}-\frac{\zeta_1}{4})(\|\bx^k\|^2_\bK+\|\bs^k\|^2_\bK)+f(\bar{x}^k)-f^*\notag\\
    \geq& \zeta_3\hat{V}^k>f(\bar{x}^k)-f^*>0,
\end{align}
where $\zeta_1=1-c_1+\sqrt{(1-c_1)^2+\theta^2}$ with $c_1=\underline{\lambda}_{\bG}(\theta+\frac{1}{\rho\bar{\lambda}_\bL})$ and $\zeta_3=\frac{1}{2}-\frac{\zeta_1}{4}$. Since $\bar{c}_\theta<\frac{1}{\kappa_\bG}$ (c.f. \eqref{range c alpha gamma}), we obtain $c_1>\theta\underline{\lambda}_\bG=\theta\bar{\lambda}_\bG/\kappa_\bG>\theta \bar{c}_\theta\bar{\lambda}_\bG=\theta^2>\frac{1}{4}\theta^2$. This implies that $\zeta_1<1-c_1+\sqrt{c_1^2-2c_1+1+4c_1}=2$, and thus $\zeta_3>0$. 

Similarly, \eqref{def V} also implies that
\begin{align}\label{Vk leq hat Vk}
    V^k\leq &\frac{1}{2}\|\bx^k\|^2_\bK+\frac{1}{2}\bar{\lambda}_{\bG}(\theta+\frac{1}{\rho\underline{\lambda}_\bL})\|\bs^k\|^2_\bK+\frac{\zeta_2}{4}\|\bx^k\|^2_\bK+\frac{\theta}{4\zeta_2}\|\bs^k\|^2_{\bK}+f(\bar{x}^k)-f^*\notag\\
    =&(\frac{1}{2}+\frac{\zeta_2}{4})(\|\bx^k\|^2_\bK+\|\bs^k\|^2_\bK)+f(\bar{x}^k)-f^*\notag\\
    \leq& \zeta_4\hat{V}^k,
\end{align}
where $\zeta_2=1-c_2+\sqrt{(c_2-1)^2+\theta^2}$ with $c_2 = \bar{\lambda}_{\bG}(\theta+\frac{1}{\rho\underline{\lambda}_\bL})$ and $\zeta_4 = \max\{\frac{1}{2}+\frac{\zeta_2}{4},1\}$.

Subsequently, since \eqref{tilde Vk+1-tilde Vk leq 0} implies that
\begin{align*}
	&V^{k+1} - V^k - (D_1\|\bw^k\|^2+D_2\|\be^k\|^2)\leq - \|\mb{x}^k\|^2_{\bar{\lambda}_{\bG}(\xi_1-\xi_2\bar{\lambda}_{\bG})\mb{K}}-\alpha(\xi_6-\xi_7\alpha)\|\bar{\bg}^k\|^2, 
\end{align*}
summing this inequality from $k=0$ to $K$ yields
\begin{align}\label{xk+gkleq}
	&\sum_{k=0}^{K}(\bar{\lambda}_{\bG}(\xi_1-\xi_2\bar{\lambda}_{\bG})\|\mb{x}^k\|^2_\bK+\alpha(\xi_6-\xi_7\alpha)\|\bar{\bg}^k\|^2)\notag\\
	\leq&V^0-V^{K+1}+\sum_{k=0}^K(D_1\|\bw^k\|^2+D_2\|\be^k\|^2) \notag\\
	\overset{\eqref{Vk geq hat Vk}}{\leq}& V^0 +\sum_{k=0}^K(D_1\|\bw^k\|^2+D_2\|\be^k\|^2)\notag\\
    \overset{\eqref{Vk leq hat Vk}}{\leq}& \zeta_4\hat{V}^0+\sum_{k=0}^K(D_1\|\bw^k\|^2+D_2\|\be^k\|^2).
\end{align}

We rewrite the optimality gap in \eqref{def hatWk} as
\begin{equation}\label{redefine hatWk}
    \hat{W}^k := \|\mb{x}^{k}-\bar{\bx}^k\|^2+ \frac{1}{N}\|\sum_{i=1}^N \nabla {f}_i(x_i^k)\|^2 \overset{\eqref{def parameter}}{=} \|\bx^k\|^2_\bK+\|\bar{\bg}^k\|^2.
\end{equation}

Incorporating \eqref{D1w+D2e lemma} into \eqref{xk+gkleq} yields
\begin{align*}
    \sum_{k=0}^K\mathbb{E}[\hat{W}^k] \overset{\eqref{redefine hatWk}}{\leq} & \frac{1}{\zeta_5}\sum_{k=0}^{K}\mathbb{E}[\big(\bar{\lambda}_{\bG}(\xi_1-\xi_2\bar{\lambda}_{\bG})\|\mb{x}^k\|^2_\bK \\
    &+\alpha(\xi_6-\xi_7\alpha)\|\bar{\bg}^k\|^2\big)]\\
    \overset{\eqref{xk+gkleq}}{\leq} & \frac{1}{\zeta_5}(\zeta_4\hat{V}^0+\mathbb{E}[\sum_{k=0}^K(D_1\|\bw^k\|^2+D_2\|\be^k\|^2)])\\
    \overset{\eqref{D1w+D2e lemma}}{\leq} & \frac{1}{\zeta_5}(\zeta_4\hat{V}^0+(D_1+D_2)\frac{2N\bar{u}^2}{1-\bar{r}^2}),
\end{align*}
where $\zeta_5 = \min\{\bar{\lambda}_{\bG}(\xi_1-\xi_2\bar{\lambda}_{\bG}),\alpha(\xi_6-\xi_7\alpha)\}$ and $\hat{V}^0=\|\bx^0-\bar{\bx}^0\|^2+\frac{1}{N}\|\sum_{i=1}^N \nabla {f}_i(x_i^0)\|^2+f(\bar{x}^0)-f^*$. Hence, we prove Theorem~\ref{theorem:nonconvex}.



\subsection{Proof of Theorem~\ref{theorem:PL}}\label{proof theorem PL}
It follows from \eqref{tilde Vk+1-tilde Vk leq 0} that 
\begin{align}\label{Vk+1 -Vk -wk-ek leq zeta Vk}
    &V^{k+1} - V^k - (D_1\|\bw^k\|^2+D_2\|\be^k\|^2)\notag\\
        \leq & - \|\mb{x}^k\|^2_{\bar{\lambda}_{\bG}(\xi_1-\xi_2\bar{\lambda}_{\bG})\mb{K} }- \|\mb{s}^k \|^2_{\bar{\lambda}_\bG^2(\xi_3-\xi_4 \bar{\lambda}_\bG-\xi_5\bar{\lambda}_\bG^2)\bK}-\frac{\alpha}{8}\|\bar{\bg}_a^k\|^2\notag\\
        \overset{\eqref{def pl}}{\leq}& - \|\mb{x}^k\|^2_{\bar{\lambda}_{\bG}(\xi_1-\xi_2\bar{\lambda}_{\bG})\mb{K} }- \|\mb{s}^k \|^2_{\bar{\lambda}_\bG^2(\xi_3-\xi_4 \bar{\lambda}_\bG-\xi_5\bar{\lambda}_\bG^2)\bK} -\frac{\alpha\nu N}{4}(f(\bar{x}^k)-f^*)\notag\\
        \overset{\eqref{range zeta 6}\eqref{def hat V}}{<} &-\zeta_6\hat{V}^k \notag\\
        \overset{\eqref{Vk leq hat Vk}}{\leq}& -\frac{\zeta_6}{\zeta_4}V^k\overset{\eqref{range zeta}}{=}-\zeta V^k,
\end{align}
where $\zeta_6=\min\{\bar{\lambda}_\bG(\xi_1-\xi_2\bar{\lambda}_\bG),\bar{\lambda}_\bG^2(\xi_3-\xi_4\bar{\lambda}_\bG-\xi_5\bar{\lambda}_\bG^2),\frac{\alpha\nu N}{4},\zeta_4(1-\bar{r}^2)\}.$
This implies that
\begin{align}\label{Vk+1 leq 1-zetak+1}
    &\mathbb{E}[V^{k+1}] \notag\\
    \leq & (1-\zeta) \mathbb{E}[V^k]+\mathbb{E}[(D_1\|\bw^k\|^2+D_2\|\be^k\|^2)]\notag\\
    \leq & (1\!-\!\zeta)^{k+1}\mathbb{E}[V^0] \!+\! \sum_{t=0}^k(1\!-\!\zeta)^{k-t}\mathbb{E}[(D_1\|\bw^t\|^2\!+\!D_2\|\be^t\|^2)]\notag\\
    \leq & (1-\zeta)^{k+1}(V^0+2N(D_1+D_2)\bar{u}^2\sum_{t=0}^k(1-\zeta)^{-t-1}\bar{r}^2t)\notag\\
    \leq & (1-\zeta)^{k+1}\Big(V^0+\frac{2N(D_1+D_2)\bar{u}^2(1-\frac{\bar{r}^2}{1-\zeta})^k}{1-\zeta-\bar{r}^2}\Big)\notag\\
    \overset{\eqref{Vk leq hat Vk}}{\leq}& (1-\zeta)^{k+1}\Big(\zeta_4\hat{V}^0+\frac{2N(D_1+D_2)\bar{u}^2}{1-\zeta-\bar{r}^2}\Big),
\end{align}
where the forth inequality holds since $1-\zeta\overset{\eqref{range zeta}}{=}1-\frac{\zeta_6}{\zeta_4}\overset{\eqref{range zeta 6}}{>}\bar{r}^2$, which also implies $1-\zeta-\bar{r}^2>0$, and thus the right-hand side of \eqref{Vk+1 leq 1-zetak+1} is positive.
It then follows from \eqref{Vk geq hat Vk} that
\begin{align*}
    &\mathbb{E}[\|\bx^k-\bar{\bx}^k\|^2+f(\bar{x}^k)-f^*]\notag\\
    \overset{\eqref{def hat V}}{\leq}&  \mathbb{E}[\hat{V}^k] \overset{\eqref{Vk geq hat Vk}}{\leq} \frac{1}{\zeta_3}\mathbb{E}[V^k]\notag\\
    \overset{\eqref{Vk+1 leq 1-zetak+1}}{\leq} & (1-\zeta)^{k}\frac{1}{\zeta_3}\Big(\zeta_4\hat{V}^0+\frac{2N(D_1+D_2)\bar{u}^2}{1-\zeta-\bar{r}^2}\Big).\notag
\end{align*}

Hence, we obtain Theorem~\ref{theorem:PL}.


\section{Proof of Differential Privacy}\label{proof of DP}

\subsection{Proof of Theorem~\ref{theorem:dp}}\label{section:appendix proof of dp}

Given that the sequence $\{\eta^k\}$ is predetermined as an input to Algorithm~\ref{algrithm DPP2}, the observation sequence $\mathcal{O}=\{\by^k,\bz^k\}_k$ is entirely determined by the noise sequences $\{\be^k\}_k$ and $\{\bw^k\}_k$. Considering the observations $\mathcal{O}^{(1)}=\{\by^{(1)k},\bz^{(1)k}\}_k$, $\mathcal{O}^{(2)}=\{\by^{(2)k},\bz^{(2)k}\}_k$ are the same, i.e., $\mathcal{O}^{(1)}=\mathcal{O}^{(2)}\in\mathcal{O}$, the dual variables $\{\bd^{(1)k},\bq^{(1)k}\}$ and $\{\bd^{(2)k},\bq^{(2)k}\}$ are completely identical as long as the initial values $\{\bd^0,\bq^0\}$ and the observable variables $\by^k$ are the same. From \eqref{yk new}, we obtain 
\begin{equation}\label{ei0k wi0k}
	\Delta e_{i_0}^k = -\Delta g_{i_0}^k, \quad \Delta w_{i_0}^k = -\Delta x_{i_0}^k,
\end{equation}
where we define
\begin{align*}
    \Delta e_{i_0}^k =& e_{i_0}^{(1)k}-e_{i_0}^{(2)k},\notag\\
    \Delta w_{i_0}^k =& w_{i_0}^{(1)k}-w_{i_0}^{(2)k},\\
    \Delta x_{i_0}^k =& x_{i_0}^{(1)k}-x_{i_0}^{(2)k},\notag\\
    \Delta g_{i_0}^k =& \nabla f_{i_0}^{(1)}(x_{i_0}^{(1)k})-\nabla f_{i_0}^{(2)}(x_{i_0}^{(2)k}).
\end{align*}
From \eqref{ei0k wi0k}, we obtain
\begin{align} \label{delta xi0k+1}
	\|\Delta x_{i_0}^{k+1}\|\overset{\eqref{xk+1 new}}{=}&\alpha\|\nabla f_{i_0}^{(1)}(x_{i_0}^{(1)k})-\nabla f_{i_0}^{(2)}(x_{i_0}^{(2)k})\|\notag\\
	= & \alpha \|\nabla f_{i_0}^{(1)}(x_{i_0}^{(1)k})-\nabla f_{i_0}^{(2)}(x_{i_0}^{(1)k}) + \nabla f_{i_0}^{(2)}(x_{i_0}^{(1)k}) - \nabla f_{i_0}^{(2)}(x_{i_0}^{(2)k})\|\notag\\
	\leq & \alpha (\|\nabla f_{i_0}^{(1)}(x_{i_0}^{(1)k})-\nabla f_{i_0}^{(2)}(x_{i_0}^{(1)k})\| + \|\nabla f_{i_0}^{(2)}(x_{i_0}^{(1)k}) - \nabla f_{i_0}^{(2)}(x_{i_0}^{(2)k})\|)\notag\\
	\overset{\eqref{ineq: def delta adjacency}\eqref{smooth tilde f}}{\leq} & \alpha(\delta+\bar{M}\|\Delta x_{i_0}^{k}\|)\notag\\
	=& \sum_{t=1}^{k+1} \alpha (\alpha\bar{M})^{t-1}\delta+(\alpha \bar{M})^{k+1}\|\Delta x_{i_0}^0\|\notag\\
	=&\frac{\alpha \delta (1-(\alpha \bar{M})^{k+1})}{1-\alpha \bar{M}}.
\end{align}
The relations in \eqref{ei0k wi0k} also imply that
\begin{align} \label{Delta wi0k=alpha Delta ei0k-1}
	\Delta w_{i_0}^k &= -\Delta x_{i_0}^k \overset{\eqref{xk+1 new}}{=} -\alpha(\nabla f_{i_0}^{(1)}(x_{i_0}^{(1)k-1})-\nabla f_{i_0}^{(2)}(x_{i_0}^{(2)k-1})) =- \alpha \Delta g_{i_0}^{k-1}=\alpha \Delta e_{i_0}^{k-1}.
\end{align}
Combining \eqref{Delta wi0k=alpha Delta ei0k-1} with \eqref{delta xi0k+1} yields
\begin{align}
	\|\Delta w_{i_0}^k\|=&\|\Delta x_{i_0}^k\|\leq \frac{\alpha \delta (1-(\alpha \bar{M})^{k-1})}{1-\alpha \bar{M}},\label{Delta wi0k}\\
	\|\Delta e_{i_0}^k\|=&\frac{1}{\alpha}\|\Delta x_{i_0}^{k+1}\|\leq \frac{ \delta (1-(\alpha \bar{M})^{k})}{1-\alpha \bar{M}}\label{Delta ei0k}.
\end{align}

We conclude that the update of $\bx^{k+1}$ in \eqref{xk+1 new} only depends on the noises $\bw^k$, $\be^k$ and the initialization $\bx^0,\bd^0,\by^0$ as well as the communication network represented by $\bL$. Hence, for a given observation $\mathcal{O}$, the objective functions and noise sequences share a bijective map. Here, we use function $\mathcal{R}_{\mathcal{F}}$ to denote the relation $\mathcal{O}^{(h)}=\mathcal{R}_{\mathcal{F}^{(h)}}(\be,\bw), h=1,2$, where $\mathcal{F}^{(h)}=\{\bx^0,\bd^0,\by^0,\bL,F^{(h)}\}$. Also, we denote $\mathcal{C}^{(1)}\triangleq\{\be^{(1)k},\bw^{(1)k}\}_{k=1}^K$ and $\mathcal{C}^{(2)}\triangleq\{\be^{(2)k},\be^{(2)k}\}_{k=1}^K$ such that $\mathcal{R}^{-1}_{\mathcal{F}^{(1)}}(\mathcal{O})\in\mathcal{C}^{(1)}$ and $\mathcal{R}^{-1}_{\mathcal{F}^{(2)}}(\mathcal{O})\in\mathcal{C}^{(2)}$. According to Definition~\ref{def DP}, we have 
\begin{align}\label{PF1/PF2}
	\frac{\mathbb{P}(F^{(1)}|\mathcal{O})}{\mathbb{P}(F^{(2)}|\mathcal{O})}=&\frac{\mathbb{P}(\mathcal{R}^{-1}_{\mathcal{F}^{(1)}}(\mathcal{O})|\mathcal{O})}{\mathbb{P}(\mathcal{R}^{-1}_{\mathcal{F}^{(2)}}(\mathcal{O})|\mathcal{O})}\notag\\
    =&\frac{\mathbb{P}(\{\be^{(1)},\bw^{(1)}\}|\{\be^{(1)},\bw^{(1)}\}\in \mathcal{C}^{(1)})}{\mathbb{P}(\{\be^{(2)},\bw^{(2)}\}|\{\be^{(2)},\bw^{(2)}\}\in \mathcal{C}^{(2)})}\notag\\
    =&\frac{\iint_{\mathcal{C}^{(1)}}f_{\be\bw}(\be^{(1)},\bw^{(1)})d\be^{(1)}d\bw^{(1)}}{\iint_{\mathcal{C}^{(2)}}f_{\be\bw}(\be^{(2)},\bw^{(2)})d\be^{(2)}d\bw^{(2)}},
\end{align}
where we define 
	$f_{\be\bw}(\be^{(h)},\bw^{(h)})=\Pi_{i=1}^{n}\Pi_{k=1}^{K}f_{L}(e_i^{(h)k},\theta_{i,k}^e)f_{L}(w_i^{(h)k},\theta_{i,k}^e), h=1,2$.

With $F^{(1)}$ and $F^{(2)}$ only differ from $f_{i_0}$, it then follows from \eqref{PF1/PF2} that
\begin{align}
	&\frac{\mathbb{P}(F^{(1)}|\mathcal{O})}{\mathbb{P}(F^{(2)}|\mathcal{O})}  = \Pi_{k=1}^{K}\frac{f_{L}(e_{i_0}^{(1)k},\theta_{i_0,k}^e)f_{L}(w_{i_0}^{(1)k},\theta_{i_0,k}^w)}{f_{L}(e_{i_0}^{(2)k},\theta_{i_0,k}^e)f_{L}(w_{i_0}^{(2)k},\theta_{i_0,k}^w)}\notag\\
	  &\leq\Pi_{k=1}^K e^{\frac{\sqrt{d}\|\Delta e_{i_0}^k\|_1}{u_{e,{i_0}}r_{i_0}^k}}\Pi_{k=1}^K e^{\frac{\sqrt{d}\|\Delta w_{i_0}^k\|_1}{u_{w,{i_0}}r_{i_0}^k}}\notag\\
	&\underset{\eqref{Delta ei0k}}{\overset{\eqref{Delta wi0k}}{\leq}} e^{\sum_{k=1}^K \sqrt{d}\Big(\frac{1-(\alpha \bar{M})^k}{\alpha u_{e,{i_0}}}+\frac{1-(\alpha \bar{M})^{k-1}}{\alpha u_{w,{i_0}}}\Big)\frac{\alpha \delta}{r_{i_0}^k(1-\alpha \bar{M})}}\notag\\
	&= e^{\sum_{k=1}^K \!\sqrt{d}\Big(\!(\frac{1}{\alpha u_{e,{i_0}}}\!+\!\frac{1}{u_{w,{i_0}}})-(\frac{1}{\alpha u_{e,{i_0}}}+\frac{1}{\alpha \bar{M}u_{w,{i_0}}})(\alpha\bar{M})^k\!\Big)\frac{\alpha \delta}{r_{i_0}^k(1\!-\!\alpha \bar{M})}}\notag\\
    &\leq e^{\sum_{k=1}^K \sqrt{d}\Big(\frac{1}{\alpha u_{e,{i_0}}}+\frac{1}{u_{w,{i_0}}}\Big)\frac{\alpha \delta}{r_{i_0}^k(1-\alpha \bar{M})}}.\notag
\end{align}
Comparing the inequality above and the definition in Definition~\ref{def DP} yields Theorem~\ref{theorem:dp}.

\subsection{Proof of Corollary~\ref{corollary:dp parameters}}
The condition in Theorem~\ref{theorem:dp} is written as 
    $$
    \underbrace{\sqrt{d}(\frac{1}{\alpha u_{e,i_0}}+\frac{1}{u_{w,i_0}})\frac{\alpha \delta}{1-\alpha \bar{M}}}_{\tilde{c}}\sum_{k=1}^K\underbrace{\frac{1}{r_{i_{0}}^k}}_{p_{i_0}^k:=\frac{1}{r_{i_{0}}^k}}\leq \epsilon_{i_0}.
    $$
    By summing $p_{i_0}^k$ from 1 to $K$, the above condition becomes 
    \begin{equation*}
        \frac{p_{i_0}(1-p_{i_0}^K)}{1-p_{i_0}}\leq\frac{\epsilon_{i_0}}{\tilde{c}},
    \end{equation*}    
    where $p_{i_0}>1$. We rewrite this as 
    \begin{equation*}
        p_{i_0}^K-(1+\frac{\epsilon_{i_0}}{\tilde{c}})p_{i_0}+\frac{\epsilon_{i_0}}{\tilde{c}}\leq 0.
    \end{equation*}
    Let $1<p_{i_0}<\frac{\epsilon_{i_0}}{\tilde{c}}$, the above condition can be satisfied by a more strict condition as 
    \begin{equation*}
        p_{i_0}^K-\frac{\epsilon_{i_0}}{\tilde{c}}p_{i_0} \leq 0.
    \end{equation*} 
    This implies that $p_{i_0}\in (1,(\frac{\epsilon_{i_0}}{\tilde{c}})^{\frac{1}{K-1}})$ with $\frac{\epsilon_{i_0}}{\tilde{c}}>1$, which is satisfied by $\alpha<(\epsilon_{i_0}-\frac{\sqrt{d}\bar{M}}{u_{e,{i_0}}})/[\delta(\frac{\sqrt{d}}{u_{w,{i_0}}}+\epsilon_{i_0})]$ with $u_{e,{i_0}}>\frac{\sqrt{d}\bar{M}}{\epsilon_{i_0}}$. Thus, we can find an $r_{i_0}\in ((\frac{\tilde{c}}{\epsilon_{i_0}})^{\frac{1}{K-1}},1)$.

			\bibliographystyle{IEEEtran}
			\bibliography{reference} 
			
		\begin{IEEEbiography}
			[{\includegraphics[width=1in,height=1.25in,clip,keepaspectratio]{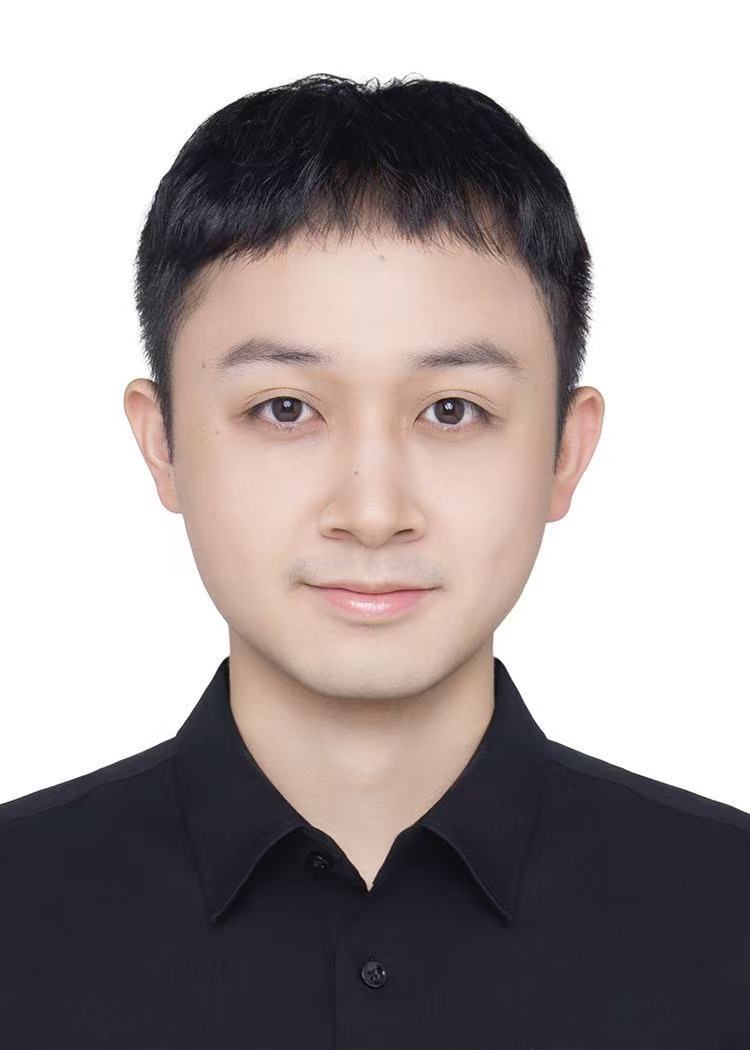}}]{Zichong Ou} received the B.S. degree in Measurement and Control Technology and Instrument from Northwestern Polytechnical University, Xi'an, China, in 2020. He is now pursuing the Ph.D degree from the School of Information Science and Technology at ShanghaiTech University, Shanghai, China. His research interests include distributed optimization, large-scale optimization, and their applications in IoT and machine learning.
		\end{IEEEbiography}
	
	\begin{IEEEbiography}[{\includegraphics[width=1.1in,height=1.22in,clip]{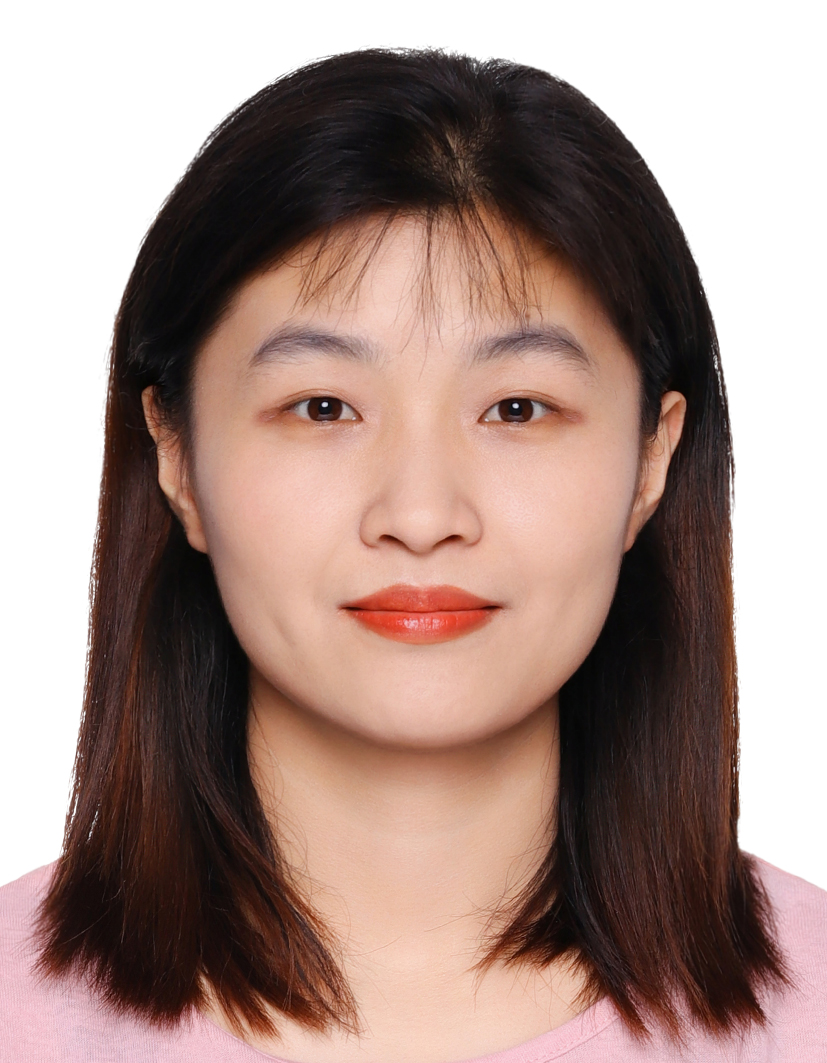}}]{Dandan Wang} received the B.S. degree in Information and Communication Engineering from Donghua University, Shanghai, China, in 2018. She is currently pursing her Ph. D. degree in the School of Information Science and Technology at ShanghaiTech University, Shanghai, China. Her research interests include distributed optimization, online optimization, and their applications in wireless networks.
	\end{IEEEbiography}

	\begin{IEEEbiography}
	[{\includegraphics[width=1in,height=1.25in,clip,keepaspectratio]{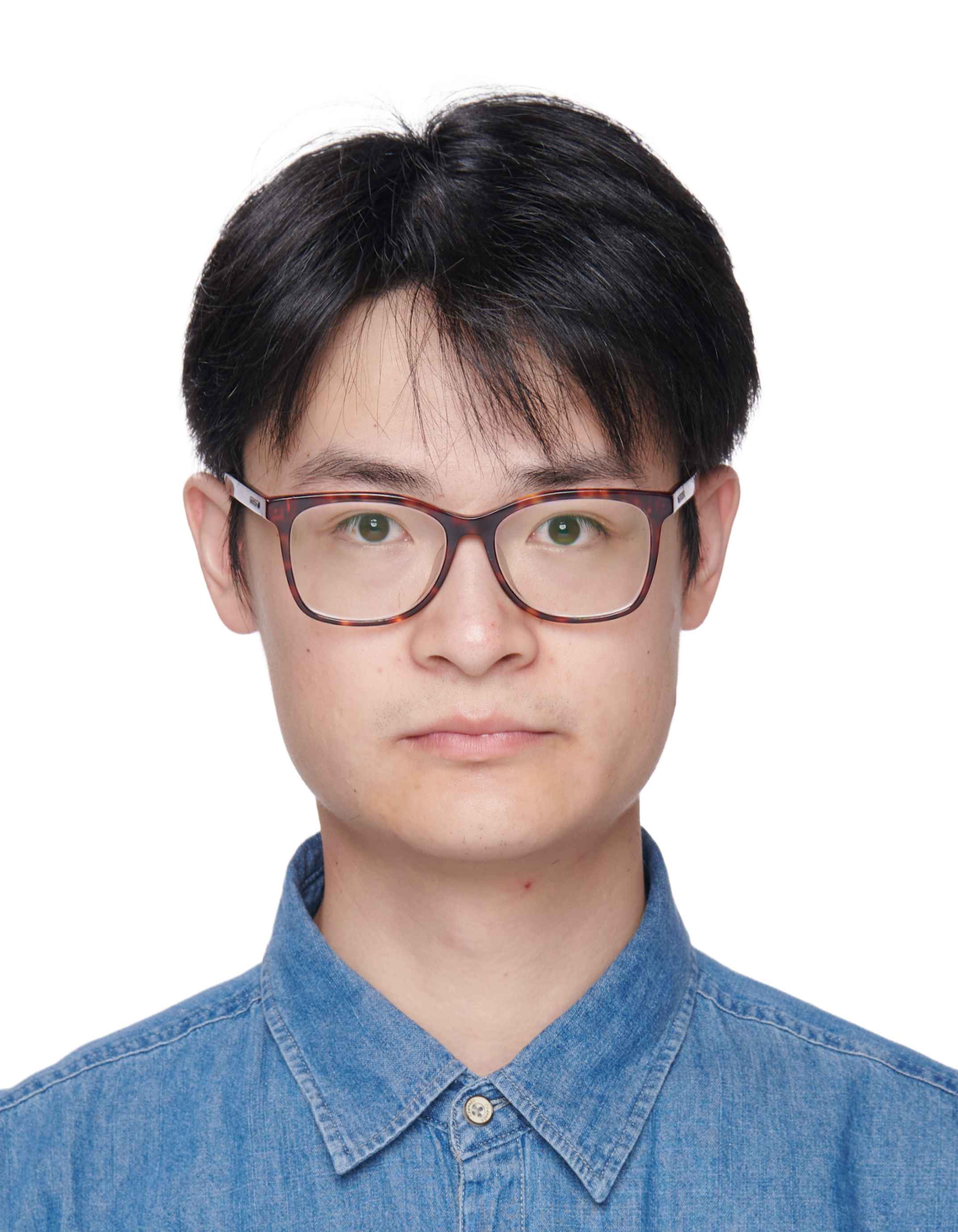}}]{Zixuan Liu} received the B.E. and M.S.E. degrees in computer science and technology from ShanghaiTech University, Shanghai, China, in 2022 and 2025, respectively. He is currently working toward the Ph.D. degree in systems and control with the Engineering and Technology Institute (ENTEG), University of Groningen, Groningen, the Netherlands. His research interests include distributed optimization and multi-agent decision making.
\end{IEEEbiography}
		
		\begin{IEEEbiography}
			[{\includegraphics[width=1in,height=1.25in,clip,keepaspectratio]{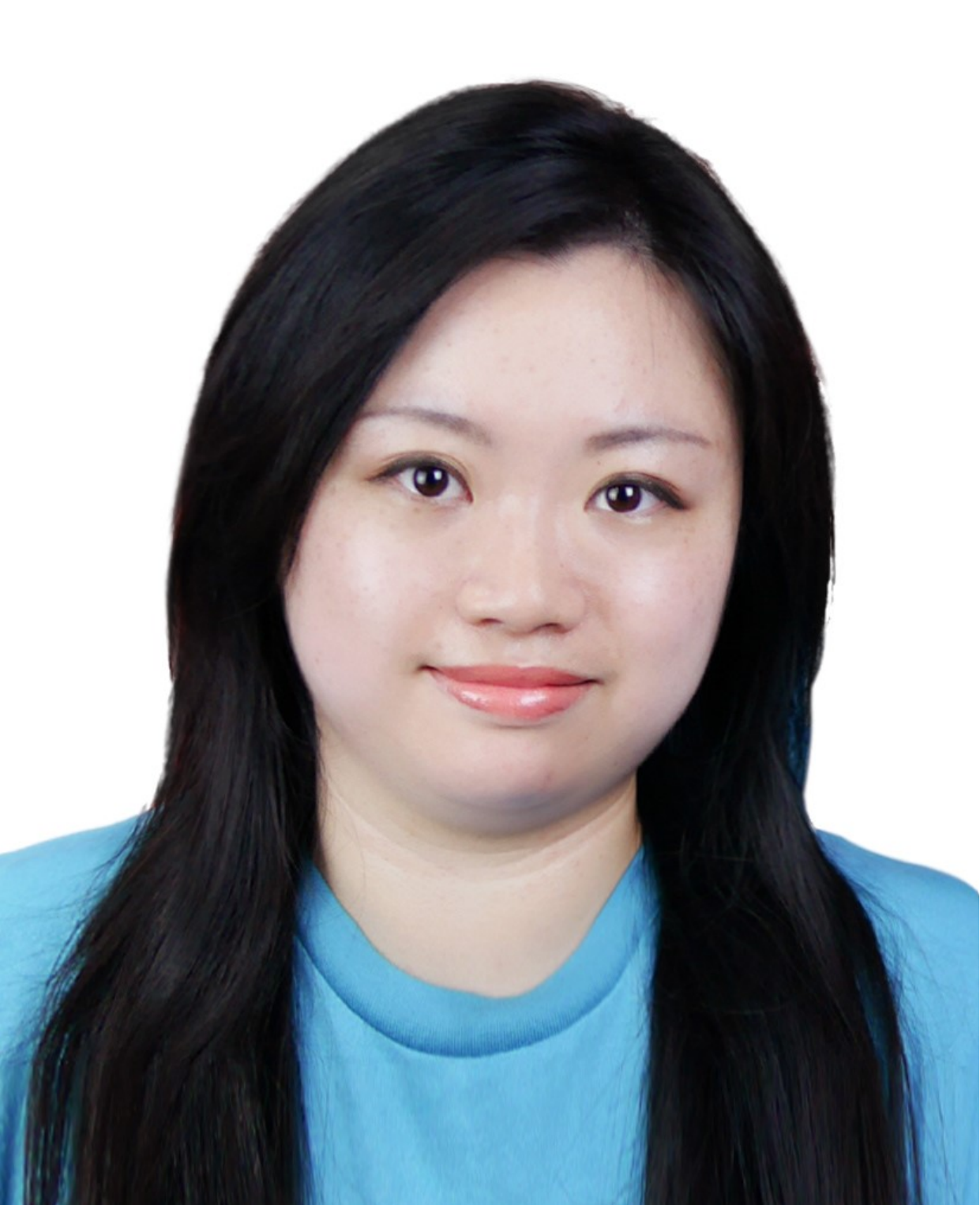}}]
			{Jie Lu} (Member, IEEE) received the B.S. degree in information engineering from Shanghai Jiao Tong University, Shanghai, China, in 2007, and the Ph.D. degree in electrical and computer engineering from the University of Oklahoma, Norman, OK, USA, in 2011. 
			She is currently an Associate Professor with the School of Information Science and Technology, ShanghaiTech University, Shanghai, China. Before she joined ShanghaiTech University in 2015, she was a Postdoctoral Researcher with the KTH Royal Institute of Technology, Stockholm, Sweden, and with the Chalmers University of Technology, Gothenburg, Sweden from 2012 to 2015. Her research interests include distributed optimization, optimization theory and algorithms, learning-assisted optimization, and networked dynamical systems.
		\end{IEEEbiography}

		\end{document}